%% file: tropical_poincare_duality.tex
\documentclass[a4paper,final]{amsart} 

\usepackage{style}      

\title{Tropical Poincaré duality spaces}
\author{Edvard Aksnes}
\address{University of Oslo \\ Moltke Moes vei 35, 0851 Oslo, Norway}
\email{\href{mailto:edvardak@math.uio.no}{edvardak@math.uio.no}}

\begin{document}
    \begin{abstract}
        The tropical fundamental class of a rational balanced polyhedral fan induces  cap products between tropical cohomology and tropical Borel--Moore homology. When all these cap products are isomorphisms, the fan is said to be a \emph{tropical Poincaré duality space}. If all the stars of faces also are such spaces, such as for fans of matroids, the fan is called a \emph{local tropical Poincaré duality space}.

        In this article, we first give some necessary conditions for fans to be tropical Poincaré duality spaces and a classification in dimension one. Next, we prove that tropical Poincaré duality for the stars of all faces of dimension greater than zero and a vanishing condition  implies tropical Poincaré duality of the fan. This leads to necessary and sufficient conditions for a fan to be a local tropical Poincaré duality space. Finally, we use such fans to show that certain abstract balanced polyhedral spaces satisfy tropical Poincaré duality.
    \end{abstract}
    \maketitle
    \input{sections/introduction}
    \input{sections/preliminaries}
    \input{sections/tropical}
    \input{sections/duality}
    \input{sections/local}

    \input{sections/polyhedral_spaces}

    \printbibliography
\end{document}

%% file: sections/introduction.tex
\section{Introduction}

For an integer $p \geq 0$, a rational polyhedral fan $\Sigma$ (\cref{def:polyhedral_fan}) and a commutative ring $R$, \cite{IKMZ} introduced the \emph{tropical homology} $H_\bullet (\Sigma, \F_p^R)$ and \emph{tropical Borel--Moore homology} $H_\bullet^{BM} (\Sigma, \F_p^R)$, along with dual constructions of \emph{tropical cohomology} $H^\bullet(\Sigma, \F_R^p)$ and \emph{tropical cohomology with compact support} $H_c^\bullet(\Sigma, \F_R^p)$, see \cref{def:tropical_homology_cohomology}. These can be computed in many different ways, see e.g. \cite{IKMZ,MZ14,JellShawSmacka,GrossShokriehSheaf}.

The \emph{balancing condition} of tropical geometry (see \cite[Definition 5.8]{BIMS14}), can be formulated homologically as the existence of a particular \emph{fundamental class} $[\Sigma,w] \in H_d^{BM}(\Sigma, \F_d^R)$ in tropical Borel--Moore homology (\cite[Proposition 4.3]{MZ14}, \cite[Remark 4.9]{Lefschetz11} and \cref{def:balanced}), depending on assigning $R$-valued weights $w$ to maximal faces. One can use the fundamental class to define a \emph{cap product}
$$\frown [\Sigma,w] \colon H^q (\Sigma,\F_R^{p})\to H_{d-q}^{BM}(\Sigma, \F_{d-p}^R)$$
for all $p,q \in \set{0,\dots,d}$, see \cite[Definition 4.11]{Lefschetz11} and \cref{def:cap_product}. If these maps are isomorphisms for all $p,q \in \set{0,\dots,d}$, one says that the fan satisfies \emph{tropical Poincaré duality over $R$} or is a \emph{tropical Poincaré duality space over $R$}, see \cref{def:tpd}. We use the abbreviation TPD for tropical Poincaré duality.

This paper, which generalizes and deepens the results from the author's master's thesis \cite{mastersthesis}, studies two questions related to tropical Poincaré duality over a given commutative ring $R$.
\begin{question}\label{question:tpd_when}
	Which fans satisfy TPD over $R$?	
\end{question}
The fan of a matroid is a TPD space over $\R$ and $\Z$ by \cite[Proposition 4.27]{JellShawSmacka} and \cite{Lefschetz11}. Moreover, motivating the question, there are fans satisfying TPD which are not fans of matroids, see \cref{ex:not_all_tpd_are_bergman_fans}.

A useful property of the cap product is that, for any commutative ring R, when it is non-zero, it is injective (\cref{prop:cap_injective}). Using this in the case where $R$ is a field, we can work with Euler characteristics and dimensions of homology groups to give a criterion for a fan to have TPD, under some vanishing assumptions (\cref{prop:euler_char_condition}).
Furthermore, we completely classify one-dimensional TPD spaces over an arbitrary commutative ring $R$.
\newtheorem*{thm:dim1}{\cref{thm:class_dim_1}}
\begin{thm:dim1}
	Let $R$ be a commutative ring, and $(\Sigma,w)$ an $R$-balanced fan of dimension one.
	Then $(\Sigma,w)$ satisfies tropical Poincaré duality over $R$ if and only if it is uniquely $R$-balanced and all the weights are units in $R$.
\end{thm:dim1}
In \cref{prop:fan_tropical_hypersurface}, we show that fan tropical hypersurfaces in $\R^n$ must have simplexes as Newton polytopes.

\begin{question}\label{question:local_tpd_when}
	Which fans satisfy TPD over $R$ at each of its faces?
\end{question}
By this, we mean that for each face $\gamma \in \Sigma$, the \emph{star fan} $\Star{\gamma}$ (\cref{def:stars_cones}) should be a TPD space over $R$. We will call this type of fans \emph{local tropical Poincaré duality spaces over R} (\cref{def:local_TPD_space}), which is equivalent to the notion of tropical smoothness defined by Amini and Piquerez \cite{AminiPiquerezFans} for $R = \Z$. Fans of matroids can be shown to be local TPD spaces.

Straddling the space between \cref{question:tpd_when} and \cref{question:local_tpd_when}, we prove the following theorem, which shows that when the stars of the faces of a fan are TPD spaces, so is the whole fan, under some vanishing conditions on Borel–-Moore homology.
\newtheorem*{thm:tpd_from_faces}{\cref{thm:stars_poincare_duality}}
\begin{thm:tpd_from_faces}
	Let $R$ be a principal ideal domain, and $(\Sigma,w)$ be an $R$-balanced fan of dimension $d\geq 2$, with $H_{q}^{BM}(\Sigma, \F_{p}^R) = 0$ for $q\not = d$, for all $p$.
	If $(\Star{\gamma},w)$ satisfies TPD over $R$, for each $\gamma\in \Sigma$ with $\Star{\gamma} \neq \Sigma$, then $(\Sigma,w)$ satisfies TPD over $R$.
\end{thm:tpd_from_faces}
Noticing the similarity of this result to the conditions for being a local TPD space, we are led to the following characterization of local TPD spaces.
\newtheorem*{thm:stars}{\cref{thm:local_tpds_char}}
\begin{thm:stars}
	Let $R$ be a principal ideal domain, and $(\Sigma,w)$ a $d$-dimensional $R$-balanced fan. Then $\Sigma$ is a local TPD space over $R$ if and only if $H_q^{BM}(\Star{\gamma}, \F_{p}^R)=0$ for all $\gamma\in \Sigma$ and $q\neq d$, and for all faces $\beta$ of codimension 1, the star fans $\Star{\beta}$ are TPD spaces over $R$.
\end{thm:stars}

In the two-dimensional case, we use \cref{thm:stars_poincare_duality} to show that, assuming the vanishing of parts of Borel–Moore homology, a fan is a TPD space if and only if it is a local TPD space, see \cref{prop:class_dim_two}. This motivates two new questions.
\newtheorem*{question:BM_vanishing_intro}{\cref{question:BM_vanishing}}
\begin{question:BM_vanishing_intro}[Geometry of BM homology vanishing]
	Let $(\Sigma,w)$ be an $R$-balanced $d$-dimensional fan. Can the fans with $H_q^{BM}(\Star{\gamma}, \F_{p}^R)=0$ for each face $\gamma\in \Sigma$, $q\neq d$ and all $p$ be geometrically characterized?
\end{question:BM_vanishing_intro}
\newtheorem*{ques:global_vs_local}{\cref{question:global_vs_local_fan}}
\begin{ques:global_vs_local}[Global versus Local TPD]
	Let $(\Sigma,w)$ be an $R$-balanced fan which satisfies TPD over $R$. Does $\Star{\gamma}$ also satisfy TPD over $R$ for each $\gamma \in \Sigma$?
\end{ques:global_vs_local}

In the final part of this paper, we turn to generalizations for \emph{rational polyhedral spaces}, see \cite{Lefschetz11,JellShawSmacka}, and \emph{abstract tropical $R$-cycle} (see \cref{def:abstract_tropical_cycle}). These can be equipped with tropical homology and cohomology groups, and a balancing condition for abstract tropical $R$-cycles leads to cap products. 
\emph{Tropical manifolds} are spaces equipped with charts to Bergman fans of matroids. These are studied in \cite{JellShawSmacka,Lefschetz11,GrossShokriehSheaf}, and are shown to satisfy TPD over $\R$ and $\Z$. Note also that for tropical Calabi-Yau hypersurfaces, in the $p=q=1$ case, there is a cap product map which can be shown to be an isomorphism by recent work of \cite{Ruddat}.

The Mayer--Vietoris arguments used in \cite{Lefschetz11} to show TPD on tropical manifolds can be applied more broadly. We say that an abstract tropical cycle is a \emph{local TPD space over $R$} if it is built from fans which are local TPD spaces over $R$. These are the building blocks of the smooth tropical cycles as defined in \cite{AminiPiquerezFans}. We then prove the following theorem.
\newtheorem*{thm:local_tpd_abstract}{\cref{thm:local_tpd_polyspaces_have_tpd}}
\begin{thm:local_tpd_abstract}
	Let $X$ be a local tropical Poincaré duality space over $R$. Then $X$ satisfies tropical Poincaré duality over $R$.
\end{thm:local_tpd_abstract}
Recently, \cite{AminiPiquerez} establishes a full ``Kähler package'' for smooth projective tropical cycles, working with rational coefficients. They relate TPD of the canonical compactifications of Bergman fans of matroids to the Poincaré duality of the \emph{Chow ring} of a matroid established in \cite{AdiprasitoHuhKatz}, which was used in proving the Heron--Rota--Welsh conjecture.
It is suggested in \cite{HuhTropGeomMatroid} that such ``Chow rings'' satisfying three properties, collectively dubbed the ``Hodge package'', should be responsible for the log-concavity of many sequences which arise in mathematics. 

In forthcoming work \cite{AAPS}, the authors show that the Tropical Poincaré duality property is a critical ingredient in relating the topology of a variety to the tropical cohomology of its tropicalization.

\subsection*{Organization}
In \cref{section:Preliminaries}, we set conventions for fans, stars and integer weights. Then we define cellular (co)sheaves and cellular (co)sheaf (co)homology.

In \cref{section:balancing}, we define the tropical multi-tangent cosheaves and sheaves, which we
use to define tropical (co)homology. This is used to describe a generalized version of the balancing condition in tropical geometry, to generalize beyond integer weights.

In \cref{section:Tropical_Poincare_duality}, we define the TPD over a ring $R$, and give some necessary conditions. Moreover, we give a complete classification in dimension one, and some criteria in codimension one of $\R^n$ for TPD to hold, which forms a first step towards answering \cref{question:tpd_when}.

In \cref{section:local}, we turn to \cref{question:local_tpd_when}. We first relate TPD at the stars of faces to TPD of the whole fan, which is then used to characterize local TPD spaces. We then use our dimension one result to give a more geometric description of the characterization. 

Finally, in \cref{section:poly_complexes}, we use local TPD spaces to construct abstract tropical cycles satisfying tropical Poincaré duality.

\subsection*{Acknowledgments}
I wish to thank Kris Shaw for the many comments, ideas and discussions which have made this article possible, as well as for supervising the master’s thesis from which it is inspired. Thank you to Cédric Le Texier and Simen Moe for many conversations and suggestions. I also would like to thank Omid Amini and Matthieu Piquerez for sharing an early draft of their article \cite{AminiPiquerezFans} and suggesting a new result. Finally, I thank the anonymous referee for the insightful suggestions that have improved this article. This research was supported by the Trond Mohn Foundation project ``Algebraic and Topological Cycles in Complex and Tropical Geometries''.

%% file: sections/preliminaries.tex
\section{Preliminaries}\label{section:Preliminaries}
In this section, we define and give references to the main objects and concepts used in the remainder of the article. In \cref{subsection:cones_fans_stars}, we introduce some conventions for weighted fans and the balancing condition, and for cellular sheaves and cosheaves in \cref{subsection:cellular_sheaves_cosheaves}. Finally, we introduce notions of homology and cohomology of cosheaves and sheaves in \cref{subsection:cellular_homology_cohomology}.
\subsection{Cones, fans and stars}\label{subsection:cones_fans_stars}
Let $N\cong \Z^n$ be a lattice, and $N_\R = N \otimes_\Z \R \cong \R^n$ be the associated real vector space.
\begin{definition} \label{def:cone}
    A \emph{rational polyhedral cone} $\sigma$ in a lattice $N$ is a set of the form  
	$$\sigma = \left\{ \textstyle{\sum_{i=1}^m} a_i v_i \mid a_i \in \Z_{\geq 0}\right\}\subset N$$ 
	for vectors $v_i \in N$, such that $\sigma_\R = \sigma \otimes_\Z \R \subset N_\R$ is closed and strictly convex, hence has a vertex at the origin.
	
	The \emph{lattice} $L_\Z (\sigma)$ is the saturated sublattice of $N$ generated by $\sigma$, and the \emph{dimension} of a cone is the rank of $L_\Z (\sigma)$.
	
	Another cone $\tau$ is said to be a \emph{face} of $\sigma$ if there is some element 
	$m\in \Hom_\Z(N,\Z)$, with $m(x)\geq 0$ for all $x\in \sigma$, i.e. a positive functional, such that $\tau=\set{x\in \sigma \mid m(x)=0}$.
	Any face can also be exhibited by setting particular coefficients $a_i$ to $0$.

	For $\tau$ a face of $\sigma$, the set $L_\Z (\tau) \subset L_\Z (\sigma)$ is a sublattice. For $\dim \tau = \dim \sigma -1$, we may select a \emph{primitive integer vector} $v_{\sigma / \tau} \in  N$ such that $L_\Z (\sigma) = L_\Z (\tau) + \Z v_{\sigma / \tau}$.
\end{definition} 
\begin{definition}\label{def:polyhedral_fan}
    A \emph{rational polyhedral fan} $\Sigma$ is a finite collection of rational polyhedral cones in $N$ such that:
    \begin{itemize}
        \item For any cone $\sigma \in \Sigma$, if $\tau$ is a face of $\sigma$, then $\tau \in \Sigma$,
        \item For $\sigma_1, \sigma_2 \in \Sigma$, the intersection $\sigma_1 \cap \sigma_2$ is a face of $\sigma_1$ and $\sigma_2$.
    \end{itemize}
    The cones in $\Sigma$ are also called \emph{faces}, and the collection of faces of dimension $i$ is denoted by $\Sigma^i$.
	The \emph{dimension of $\Sigma$} is the supremum of the dimensions of cones of $\Sigma$. 
    We write $\tau \preceq \sigma$ if $\tau$ is a face of $\sigma$ and $\tau \prec \sigma$ if $\tau$ is a proper face, which gives a partial ordering on $\Sigma$. We say that a face $\sigma \in \Sigma$ is \emph{maximal} if it is maximal with respect to the ordering $\preceq$.  We will require that all fans are \emph{pure dimensional} in the sense that all maximal by inclusion faces are of equal dimension.
\end{definition}
Abusing notation, we also write $\Sigma$ for the category associated to the partial ordering $\preceq$, whose objects are the cones $\sigma \in \Sigma$, with a morphism $\tau \to \sigma \in \Hom_\Sigma (\tau, \sigma)$ if and only if $\tau \preceq \sigma$. 

Note that all cones intersect in a common minimal cell, and since we required each cone to have a vertex, this is the unique vertex $v$ in $\Sigma$. Moreover, a rational polyhedral fan corresponds to a \emph{cell complex} in the sense of \cite{Shepard,Curry}, when considering the fan as glued abstractly from the interiors of the cones.
\begin{example}\label{ex:cross}
	Consider the fan $\Sigma$ displayed in \cref{fig:cross}, which consists of the rays $\tau_1, \tau_2, \tau_3, \tau_4$ and the vertex $v$. The fan is pure dimensional, its maximal faces are the $\tau_i$ and it has dimension $1$. It consists of the union of the line $x=0$ and $y=0$ when considered in $N_\R$. We have $v \prec \tau_i$ for each $i$.
\end{example}
\begin{example}\label{ex:complete}
	Another example of a fan is shown in \cref{fig:complete}. This fan has one vertex $v$, three one-dimensional cones $\tau_i$, and three two-dimensional cones $\sigma_i$. For instance, the faces of $\sigma_1$ are the cones $\tau_1$ and $\tau_2$ as well as the vertex $v$.
\end{example}
\begin{figure}[h]
	\centering
	\begin{minipage}[b]{.5\textwidth}
		\centering
		\input{figures/cross.tex}
		\captionof{figure}{The cross} 
		\label{fig:cross}
	\end{minipage}%
	\begin{minipage}[b]{.5\textwidth}
		\centering
		\input{figures/complete.tex}
		\captionof{figure}{The complete fan} 
		\label{fig:complete}
	\end{minipage}
\end{figure}

Fans of particular interest in tropical geometry are the \emph{Bergman fans} of matroids (see \cite{ZharkovOrlikSolomonAlg,BergmanComplexPhyloTree} for definitions). These serve as the local models of abstract tropical manifolds (see \cite[Section 1.6]{MZ14}).
\begin{example}
	Let $M$ be a matroid on $E=\set{0,\dots, n}$ with lattice of flats $\mathscr{L}$, and let $\Z \set{e_0, \dots, e_n}$ be the lattice of rank $n+1$ generated by elements $e_0, \dots, e_n$. Let $N$ be the quotient defined by 
	\[\begin{tikzcd}[column sep = small]
		0 & {\Z\{e_0 + \dots + e_n\}} & {\Z \{e_0, \dots, e_n\}} & N & 0.
		\arrow[from=1-1, to=1-2]
		\arrow[from=1-2, to=1-3]
		\arrow["\pi", from=1-3, to=1-4]
		\arrow[from=1-4, to=1-5]
	\end{tikzcd}\]
	
	For any subset $S\subseteq E$, let 
	$p_S = \sum_{i \in S} \pi(e_i)$
	in $N$, so that in particular $p_E = 0$. For any chain $F_\bullet$ of flats of the matroid $M$,
	$$F_\bullet = \set{\emptyset \subsetneq F_1 \subsetneq \dots \subsetneq F_k \subsetneq E} \subseteq \mathscr{L}.$$
	the cone associated to $F_\bullet$ is the non-negative span 
	$$\sigma(F_\bullet) = \left\{ \sum_{i=1}^k a_i p_{F_i} \mid a_i \geq 0, \, i = 1,\dots, k \right\}.$$
	The Bergman fan of $M$ is the simplicial fan $\Sigma(M)$ consisting of cones $\sigma(F_\bullet)$ for all flags of flats $F_\bullet$.

	The $U_{3,4}$ matroid on the set $E=\set{0,\dots,3}$ given by the rank function $r\colon 2^E \to \Z_{\geq 0}$ taking values $r(S)= \min(|S|,3)$ has the lattice of flats given by \cref{fig:lattice_of_flats_u34}. The Bergman fan of this matroid is shown in \cref{fig:bergman_fan_u34}.
	\begin{figure}
		\centering
		\begin{minipage}[b]{.5\textwidth}
			\centering
			\includegraphics[width=.6\linewidth]{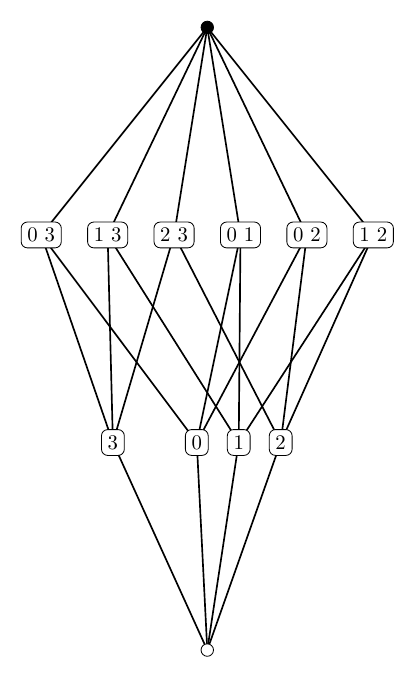}
			\captionof{figure}{Lattice of flats of the $U_{3,4}$ matroid.} 
			\label{fig:lattice_of_flats_u34}
		\end{minipage}%
		\begin{minipage}[b]{.5\textwidth}
			\centering
			\includegraphics[width=.8\linewidth]{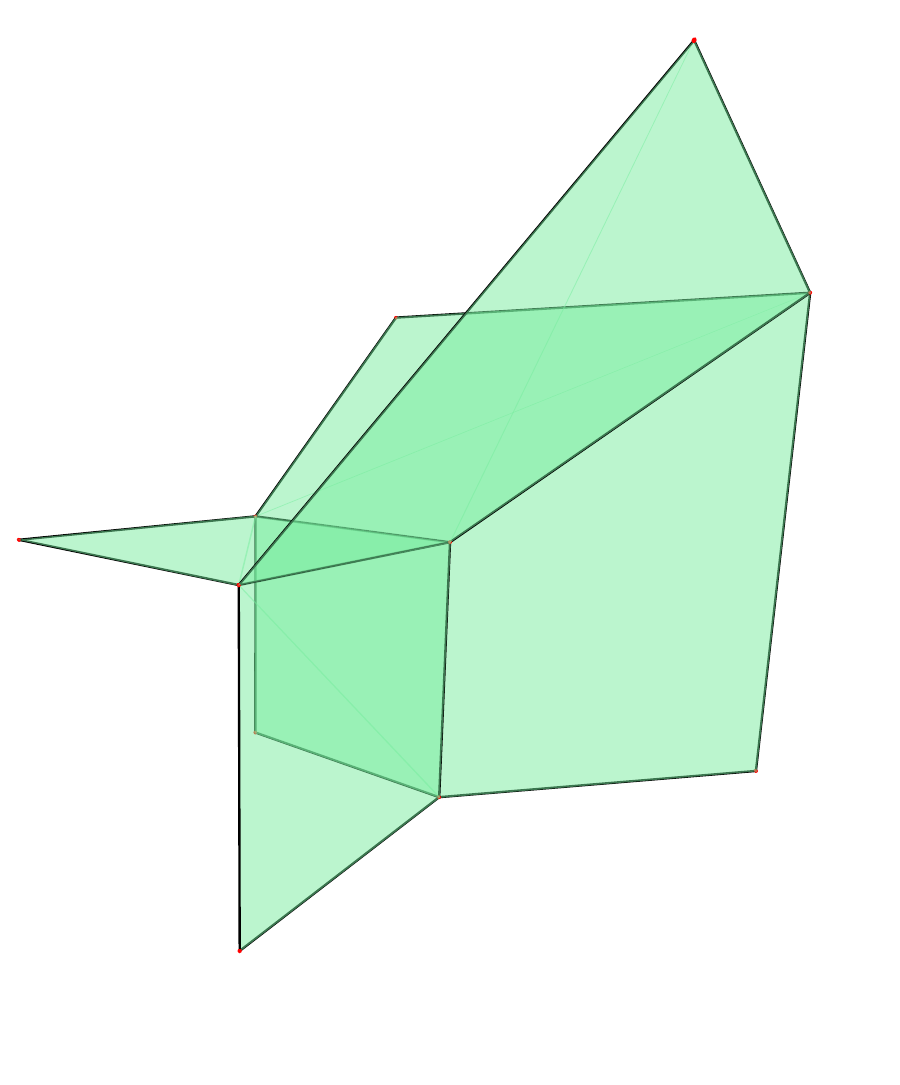}
			\captionof{figure}{Bergman fan of the $U_{3,4}$ matroid, visualization by \cite{polymake}.}
			\label{fig:bergman_fan_u34}
		\end{minipage}
	\end{figure}
\end{example}
\begin{definition}\label{def:stars_cones}
	The \emph{star} $\Star{\gamma}$ at a cone $\gamma \in \Sigma$ is the rational polyhedral fan with underlying set $\cup_{\gamma \preceq \kappa} \widetilde{\kappa}\subset N$, where $\widetilde{\kappa}=\set{t (x-y) \mid t\in \Z_{\geq 0}, x \in \kappa, y \in \gamma}\subseteq N$, subdivided into rational polyhedral cones with a shared vertex.

	The \emph{cone} at a face $\gamma \in \Sigma$ is the fan $\Cone{\gamma}$ consisting of the faces $\kappa \in \Sigma$ for each $\kappa \preceq \gamma$. In particular the vertex $v$ of $\Sigma$ is the minimal cell in each cone.
\end{definition}
\begin{example} 
	We give two examples of stars:
	\begin{enumerate}
		\item In the fan from \cref{ex:complete}, the cone $\tau_1$ is contained in the cones $\sigma_1$ and $\sigma_2$. These give rise to the sets $\widetilde{\sigma_1}=\set{(a,b)\in N \mid a\geq 0}$ and $\widetilde{\sigma_2}=\set{(a,b)\in N \mid a\leq 0}$, so that the star $\Star{{\tau_1}}$ has underlying set equal to the whole of $N$. 
		\item In the Bergman fan of the $U_{3,4}$ matroid, the star at any of the one-dimensional rays is has underlying set equal to a product of $\Z$ together with the ``tropical line'', i.e. the fan with rays $(1,1), (-1,0),(0,-1)$ and a vertex at $(0,0)$.
	\end{enumerate}
	In both cases, these sets must then be cut up so as to form a rational polyhedral fan.
\end{example}

An \emph{integer weight function} on a rational polyhedral fan $\Sigma$ of dimension $d$ is a function $w \colon \Sigma^{d} \to \Z$. We are interested in weighted fans satisfying the usual tropical \emph{balancing condition}. This condition is equivalent to being a \emph{Minkowski weight} in the sense of \cite{FultonSturmfels}. For more on the balancing condition, see for instance \cite[Definition 3.3.1]{MaclaganSturmfels} or \cite[Definition 5.7]{BIMS14}.

\begin{definition}\label{def:balancing_Z}
	Let $\Sigma$ be a rational polyhedral fan of dimension $d$ with weights $w \colon \Sigma^{d} \to \Z$.
	We say that $\Sigma$ together with $w$ is \emph{balanced} at a face $\beta \in \Sigma^{d-1}$ if $$\sum_{\beta \prec \alpha} w(\alpha) v_{\alpha/\beta} \in L_\Z (\beta),$$
	using the notation from \cref{def:cone}.
	We say $\Sigma$ together with $w$ is \emph{balanced} if it is balanced along each face $\beta \in \Sigma^{d-1}$. 
\end{definition}
\begin{example}\label{ex:cross_balancing}
	Our previous examples of fans have all been balanced:
	\begin{enumerate}
		\item The fan of dimension one discussed in \cref{ex:cross} and shown in \cref{fig:cross} is balanced, for a given weight function $w \colon \Sigma^1 \to \Z$, if and only if $w(\tau_1)=w(\tau_3)$ and $w(\tau_2)= w(\tau_4)$. 
		\item The fan of dimension two in \cref{ex:complete} is also balanced if and only if the weight function $w \colon \Sigma^1 \to \Z$ is such that $w(\sigma_1)=w(\sigma_2)=w(\sigma_3)$.
		\item It follows from \cite[Proposition 2]{BergmanComplexPhyloTree}, that the stars $\Star{\gamma}$ of faces $\gamma$ in the Bergman fans of a matroid are themselves Bergman fans of matroids. It is shown in \cite[Proposition 5.2]{AdiprasitoHuhKatz} that, for the Bergman fan $\Sigma(M)$ of a matroid $M$, the only weight functions which satisfy the balancing condition are the constant ones. The uniqueness of such a weight function follows from tropical Poincaré duality in \cite[Proposition 5.5]{Lefschetz11} and the earlier \cite[Theorem 38]{HuhThesis}. By our later \cref{def:balanced}, this will mean that these fans are \emph{uniquely $\Z$-balanced}.
	\end{enumerate}
\end{example}

\subsection{Cellular sheaves and cellular cosheaves}\label{subsection:cellular_sheaves_cosheaves}
One can define \emph{cellular sheaves} and \emph{cellular cosheaves} of modules on a polyhedral fan:
\begin{definition}\label{definition:sheaf_cosheaf}
	Let $R$ be a commutative ring, $\Sigma$ a rational polyhedral fan. Then:
	\begin{itemize}
		\item A \emph{cellular $R$-sheaf} $\G$ is a functor $\G \colon \Sigma \to \Mod_R$.
		\item A \emph{cellular $R$-cosheaf} $\F$ is a functor $\F \colon \Sigma^\op \to \Mod_R$.
	\end{itemize}
	A morphism of sheaves or cosheaves is simply a natural transformation of functors or contravariant functors, respectively.
\end{definition}
\begin{remark}
	The category $\Sigma$, when viewed as a set, can be given the \emph{Alexandrov topology}, such that cellular sheaves and cosheaves in fact are sheaves and cosheaves with respect to this topology. For more on cellular sheaves and cosheaves, see \cite{Curry}.

	We have considered the fan $\Sigma$ as a category with morphisms $\tau \to \sigma$ whenever $\tau$ is a face of $\sigma$, so that a sheaf $\G$ induces a map $\G(\tau) \to \G(\sigma)$, and a cosheaf $\F$ induces a map $\F(\sigma) \to \F (\tau)$. This convention is in agreement with \cite{Curry,Shepard}, but reversed from \cite{Brion,Hower} in the sense that their \emph{sheaves} are our \emph{cosheaves}, and vice versa.
\end{remark}
\begin{example}\label{ex:constant_sheaf_cosheaf}
	Let $\Sigma$ be a rational polyhedral fan. For a module $M$ over a ring $R$, the \emph{constant cosheaf $M^\Sigma$ with values in $M$} is the cosheaf defined as a functor $M^\Sigma \colon  \Sigma^\op \to \mathbf{Mod}_R$ taking all objects to $M$ and all morphisms to $\id_M$.
	
	Similarly, the \emph{constant sheaf $M_\Sigma$ with values in $M$} is the sheaf defined as a functor $M_\Sigma \colon  \Sigma \to \mathbf{Mod}_R$ taking all objects to $M$ and all morphisms to $\id_M$.
\end{example}

\subsection{Cellular homology and cohomology}\label{subsection:cellular_homology_cohomology}
Considering the fan $\Sigma$ as a subset of $N_\R$, we select an orientation for each cone $\sigma \in \Sigma$. For each $\tau \prec \sigma$ such that $\dim (\tau) = \dim (\sigma)- 1$, we keep track of the relative orientations by writing $\mathcal{O}(\tau, \sigma)=1$ if the restriction of the orientation of $\sigma$ to $\tau$ coincides with the orientation of $\tau$, and $\mathcal{O}(\tau, \sigma)=-1$ if it reverses it. In the two next definitions, we use the orientation $\mathcal{O}(\tau,\sigma)=\pm 1$ to construct certain (co)chain complexes for a given (co)sheaf. These definitions are equal to the ones in \cite{Curry, Shepard,CellularSheaves}, and reversed from \cite{Brion, Hower}, who index by codimension.
\begin{definition}\label{def:cochain_group}
	Given a cellular sheaf $\G$, the \emph{cellular cochain groups} and \emph{cellular cochain groups with compact support} are defined, respectively, by
	\begin{equation*}
		C^q(\Sigma,\G) \coloneqq \bigoplus_{\substack{\sigma \in \Sigma^q \\ \sigma_\R \; \text{compact}}} \G(\sigma)
		\quad \text{and} \quad
		C_c^q(\Sigma,\G) \coloneqq \bigoplus_{\sigma \in \Sigma^q} \G(\sigma),
 	\end{equation*}
 	for $q\geq 0$, where $\sigma_\R$ is as in \cref{def:cone}. The cellular cochain maps
 	\begin{equation*}
 		d^q \colon C^q (\Sigma,\G) \to C^{q+1} (\Sigma,\G) 
		\quad \text{and} \quad
		d^q \colon C_c^q (\Sigma,\G) \to C_c^{q+1} (\Sigma,\G) 
 	\end{equation*}
 	are given component-wise for $\tau \in \Sigma^q$ and $\sigma \in \Sigma^{q+1}$ with $\tau \prec \sigma$ by $d_{\tau \sigma} \colon \G(\tau) \to \G(\sigma)$, where
 	\begin{equation*}
 		d_{\tau \sigma} \coloneqq \mathcal{O}(\tau, \sigma) \G(\tau \to \sigma).
 	\end{equation*}
	If $\tau \not \preceq \sigma$, we let the map $d_{\tau \sigma}$ be $0$.

	The cohomology groups $H^\bullet (\Sigma, \G)$ and $H_c^\bullet (\Sigma, \G)$ of these complexes are the \emph{cellular sheaf cohomology} and \emph{cellular sheaf cohomology with compact support} with respect to the sheaf $\G$.
\end{definition}

\begin{definition}\label{def:chain_group}
	Given a cellular cosheaf $\F$, the \emph{cellular chain group} and \emph{Borel--Moore cellular chain groups} are defined, respectively, by
	\begin{equation*}
		C_q(\Sigma,\F) \coloneqq \bigoplus_{\substack{\sigma \in \Sigma^q \\ \sigma_\R \; \text{compact}}} \F(\sigma)
		\quad \text{and} \quad
		C_q^{BM}(\Sigma,\F) \coloneqq \bigoplus_{\substack{\sigma \in \Sigma^q}} \F(\sigma),
 	\end{equation*} 
 	for $q \geq 0$, where $\sigma_\R$ is as in \cref{def:cone}. The cellular chain maps
 	\begin{equation*}
 		\partial_q \colon C_q (\Sigma,\F) \to C_{q-1} (\Sigma,\F) 
		\quad \text{and} \quad
		\partial_q \colon C_q^{BM} (\Sigma,\F) \to C_{q-1}^{BM} (\Sigma,\F) 
 	\end{equation*}
 	are given component-wise for $\sigma \in \Sigma^q$ and $\tau \in \Sigma^{q-1}$ by $\partial_{\sigma \tau} \colon \F(\sigma) \to \F(\tau)$, where
 	\begin{equation*}
 		\partial_{\sigma \tau} \coloneqq \mathcal{O}(\tau, \sigma) \F(\sigma \to \tau).
 	\end{equation*}
	If $\tau \not \preceq \sigma$, we let the map $\partial_{\sigma \tau} $ be $0$.

	The homology groups $H_\bullet (\Sigma,\F)$ and $H_\bullet^{BM} (\Sigma,\F)$ of these complexes are the \emph{cellular cosheaf homology} and \emph{cellular Borel--Moore cosheaf homology} with respect to $\F$.
\end{definition}
Proofs that the cellular (co)chain groups and maps defined above form (co)chain complexes can be found in \cite[Definitions 6.2.6-7]{Curry} and \cite[Theorem 1.1.3]{Shepard}.
\begin{remark}\label{remark:trivial_hom_cohom}
	The above definitions of cellular cohomology work in the more general setting of polyhedral complexes. Since we are working with pointed polyhedral fans, the unique compact cell is the vertex $v$. Then, for any sheaf $\G$ on a fan $\Sigma$, the cellular cochain groups $C^q(\Sigma,\G)$ are trivial for $q>0$, and therefore:
	\begin{align*}
		H^q (\Sigma, \G) = \begin{cases*}
			\G(v) & for $q=0$, \\
			0 & otherwise.
		\end{cases*}
	\end{align*}
	Similarly, for any cosheaf $\F$, the cellular chain groups $C_q(\Sigma,\F)$ are trivial for $q>0$, thus:
	\begin{align*}
		H_q (\Sigma, \F) = \begin{cases*}
			\F(v) & for $q=0$, \\
			0 & otherwise.
		\end{cases*}
	\end{align*}
\end{remark}
\begin{example}\label{ex:cross_Z_BM_homology}
	Consider the fan from \cref{ex:cross}, with orientations chosen so that $\mathcal{O}(v,\tau_i)=1$ for all $i$. The Borel--Moore homology of the constant cosheaf $\Z^\Sigma$ is the homology of the complex 
	\[\begin{tikzcd}
		0 & {\Z^4} & \Z & 0,
		\arrow[from=1-1, to=1-2]
		\arrow["\partial_1",from=1-2, to=1-3]
		\arrow[from=1-3, to=1-4]
	\end{tikzcd}\]
	where the matrix $\partial_1$ is indexed by the $\tau_i$ and given by
	$$\partial_1 = (\mathcal{O}(v,\tau_i) \id_\Z)_{\tau_i \in \Sigma^1}=(1\; , 1\; , 1\; , 1).$$
	The Borel--Moore homology becomes $H_1^{BM}(\Sigma,\Z^\Sigma)=\Z^3$ and $H_0^{BM}(\Sigma,\Z^\Sigma)=0$.
\end{example}

%% file: figures/cross.tex
\begin{tikzpicture}
    
    \foreach \X in {-2,...,2}{
        \foreach \Y in {-2,...,2}{
            \draw[fill=black] (\X,\Y) circle (1pt);
        }
    }

    \node at (-0.2,-0.2) {$v$};
    \node at (1.5,-0.3) {$\tau_1$};
    \node at (-0.3,1.5) {$\tau_2$};
    \node at (-1.5,-0.3) {$\tau_3$};
    \node at (-0.3,-1.5) {$\tau_4$};
    
    \draw (0,0) -- (2,0);
    \draw (0,0) -- (-2,0);
    \draw (0,0) -- (0,2);
    \draw (0,0) -- (0,-2);

    \draw[thick, ->] (0,0) -- (0.97,0);
    \draw[thick, ->] (0,0) -- (-0.97,0);
    \draw[thick, ->] (0,0) -- (0,0.97);
    \draw[thick, ->] (0,0) -- (0,-0.97);

\end{tikzpicture}

%% file: figures/complete.tex
\begin{tikzpicture}
    \fill [gray, fill opacity=0.5] (0,0) -- (2,0) -- (2,2) -- (0,2) -- cycle;
    
    \fill [gray, fill opacity=0.5] (0,0) -- (2,0) -- (2,-2) -- (-2,-2) -- cycle;
    
    \fill [gray, fill opacity=0.5] (0,0) -- (0,2) -- (-2,2) -- (-2,-2) -- cycle;
    
    \foreach \X in {-2,...,2}{
        \foreach \Y in {-2,...,2}{
            \draw[fill=black] (\X,\Y) circle (1pt);
        }
    }

    \node at (-0.2,0) {$v$};
    \node at (1.5,-0.3) {$\tau_1$};
    \node at (-0.3,1.5) {$\tau_2$};
    \node at (-1.5,-1.3) {$\tau_3$};
    
    \node at (1.5, 1.5) {$\mathbf{\sigma_1}$};
    \node at (0.5,-1.5) {$\mathbf{\sigma_2}$};
    \node at (-1.5,0.5) {$\mathbf{\sigma_3}$};
    
    \draw (0,0) -- (2,0);
    \draw (0,0) -- (0,2);
    \draw (0,0) -- (-2,-2);

    \draw[thick, ->] (0,0) -- (0.97,0);
    \draw[thick, ->] (0,0) -- (0,0.97);
    \draw[thick, ->] (0,0) -- (-0.97,-0.97);

\end{tikzpicture}

%% file: sections/tropical.tex
\section{Tropical geometry of fans}\label{section:balancing}
In this section, we introduce particular cellular (co)sheaves on fans which are of interest in tropical geometry. After examining some properties of the resulting \emph{tropical (co)homology}, we use this to define the \emph{balancing} condition in tropical geometry. Finally, we define the \emph{tropical cap product} associated to a balancing of the fan. We then introduce particular sheaves of interest in tropical geometry. Next, we generalize the balancing condition on fans to weights in arbitrary rings, which finally leads to a treatment of tropical Poincaré duality over arbitrary commutative rings.

\subsection{Tropical sheaves and cosheaves}
For tropical (co)homology, the following sheaves and cosheaves are of interest.
\begin{definition}[{\cite[Definition 13]{IKMZ}}] \label{def:Fp_sheaves}
	Let $\Sigma$ be a fan of dimension $d$ in $N$. For $\sigma \in \Sigma$, let $L_\Z(\sigma)$ be the lattice of integer points parallel to the cone $\sigma$. For $p=1,\dots,d$, the \emph{$p$-th multi-tangent cosheaf} $\F_p^\Z$ is the cellular $\Z$-cosheaf defined by the data: 
	\begin{itemize}
		\item For $\sigma\in \Sigma$, one has $\F_p^\Z(\sigma) \coloneqq \sum_{\sigma\preceq \gamma} \bigwedge^p L_\Z (\gamma) \subseteq \bigwedge^p N$.
		\item For $\tau \preceq \sigma$, the morphism $(\sigma \to \tau)\in \Hom_{\Sigma^\op}(\sigma, \tau)$ becomes the map $\iota_{\sigma,\tau}\colon \F_p^\Z(\sigma) \to \F_p^\Z(\tau)$, which is induced by the natural inclusion.
	\end{itemize}
	In the $p=0$ case, we define $\F_0^\Z=\Z^\Sigma$, with all maps being the identity.

	Furthermore, the cellular cosheaf $\F_p^\Z$ also gives rise to a cellular sheaf
	$\F_\Z^p$ which is defined by $\F_\Z^p (\sigma) \coloneqq \F_p^\Z (\sigma)^*$, with morphisms $\rho_{\tau,\sigma}\colon \F_\Z^p(\tau) \to \F_\Z^p(\sigma)$ defined by dualizing $\iota_{\sigma,\tau}\colon \F_p^\Z(\sigma) \to \F_p^\Z(\tau)$. 

	Finally, following \cite{Lefschetz11}, for any ring $R$, we define a cosheaf $\F_p^R$ by taking the tensor product $\F_p^R(\sigma)=\F_p^\Z(\sigma) \otimes_\Z R$, giving an $R$-module, and tensoring the maps as well. 
    Dualizing yields a sheaf $\F_R^p$.
\end{definition}
\begin{example}
	We compute some values of these cosheaves:
	\begin{enumerate}
		\item For \cref{ex:cross}, taking the ray $\tau_1$, we have that $\F_1^\Z(\tau) = L_\Z(\tau) = \langle (1,0) \rangle_\Z \subset \Z^2$. For the central vertex $v$, we have
		\begin{align*}
			\F_1^\Z (v) &= \sum_{i=1}^4 L_\Z(\tau_i) \\ &= \langle (1,0) \rangle_\Z + \langle (0,1) \rangle_\Z +  \langle (-1,0) \rangle_\Z +  \langle (0,-1) \rangle_\Z  = \Z^2.
		\end{align*}
		The cosheaf $\F_0^\Z$ is merely the constant cosheaf taking value $\Z$, so that $\F_0^\Z(\tau_i)=\Z$ and $\F_0^\Z(v)=\Z$.
		\item For \cref{ex:complete}, we have that $\F_2^\Z(\sigma_1)= \bigwedge^2 L_\Z(\sigma_2) = \langle (1,0)\wedge (0,1) \rangle_\Z \cong \Z$.
	\end{enumerate}
	
\end{example}

\begin{remark}\label{remark:change_coefficients}
	For any $\Z$-module $M$ and commutative ring $R$, the product $M_R \coloneqq M \otimes_\Z R$ is an $R$-module. Moreover, by \cite[Proposition III.7.5.8]{Bourbaki}, we have $\bigwedge^p M_R\cong (\bigwedge^p M) \otimes_\Z R$. In particular, for $\sigma \in \Sigma$ a maximal face of a $d$-dimensional fan, $L_R(\sigma)\coloneqq L_\Z(\sigma) \otimes_\Z R$ is a free $R$-module of dimension $d$, and $\F_R^p(\sigma)= (\bigwedge^p L_\Z(\sigma))\otimes_\Z R \cong \bigwedge^p L_R(\sigma)$.
\end{remark}
\begin{remark}\label{remark:Lambda_generator}
	Let $\Sigma$ be a fan of dimension $d$. For any $\alpha \in \Sigma^d$, by \cref{remark:change_coefficients} we have $\F_d^R (\alpha) = \bigwedge^d L_R (\alpha) \cong R$. Given a choice of orientation for $\alpha$, we can select the unique generator $\Lambda_\alpha \in \F_d^\Z (\alpha)= \bigwedge^p L_\Z(\sigma)$ compatible with the chosen orientation, and abusing notation, we let $\Lambda_\alpha \in \F_d^R (\alpha)\cong (\bigwedge^p L_\Z(\sigma))\otimes_\Z R $ denote the corresponding element $\Lambda_\alpha\otimes 1_R \in \F_d^R (\alpha)$.
\end{remark}
\begin{example}\label{ex:complete_lambda}
	In \cref{ex:complete}, suppose we choose orientations such that all the one-dimensional rays point outward, with all the two-dimensional cones being oriented clockwise. Choose the standard basis $e_1,e_2$ for the ambient lattice $N$. We then have $\Lambda_{\sigma_1}=e_1\wedge e_2$, $\Lambda_{\sigma_2}=e_1 \wedge e_2$ and $\Lambda_{\sigma_3} = e_1\wedge e_2$.
\end{example}
\begin{definition}\label{def:tropical_homology_cohomology}
    The cochain complex $(C^\bullet(\Sigma,\F_R^p),\delta)$ from \cref{def:cochain_group} has cohomology groups $H^q(\Sigma,\F_R^p)$ which are called the (cellular) \emph{tropical cohomology groups with $R$-coefficients} of $\Sigma$.
    Moreover, the cohomology groups $H_c^q(\Sigma,\F_R^p)$ of the complex $(C_c^\bullet(\Sigma,\F_R^p),\delta)$ are called the (cellular) \emph{compact support tropical cohomology groups with $R$-coefficients} of $\Sigma$.

	Similarly, the chain complex $(C_\bullet(\Sigma,\F_p^R),\partial)$ from \cref{def:chain_group} has homology groups $H_q(\Sigma,\F_p^R)$ which are called the (cellular) \emph{tropical homology groups with $R$-coefficients} of $\Sigma$.
	Finally, the (cellular) \emph{tropical Borel--Moore homology groups with $R$-coefficients} $H_q^{BM}(\Sigma,\F_p^R)$ are the homology groups of the chain complex $(C_\bullet^{BM}(\Sigma,\F_p^R),\partial)$.
\end{definition}

\begin{proposition}\label{prop:fp_cohom_trivial}
	The tropical cohomology with $R$-coefficients of any fan $\Sigma$ is
	\begin{align*}
		H^q (\Sigma, \F_R^p) = \begin{cases*}
			\F_R^p (v) & for $q=0$, \\
			0 & otherwise,
		\end{cases*}
	\end{align*}
	where $v\in \Sigma$ is the vertex.
\end{proposition}
\begin{proof}
	This follows from \cref{remark:trivial_hom_cohom}.
\end{proof}
\begin{example}\label{ex:cross_homology}
	Consider again the 1-dimensional fan from \cref{ex:cross}. Since it is of dimension $1$, the only $\F_R^p$ sheaves are $\F_R^0 \cong R_\Sigma$ and $\F_R^1$. By \cref{prop:fp_cohom_trivial}, the only non-zero cohomology groups are $H^0(\Sigma,\F_R^0)=R$ and $H^0(\Sigma,\F_R^1)=\F_R^1(v)=R^2$.

	Similarly, the only $\F_p^R$ cosheaves are $\F_0^R \cong R^\Sigma$ and $\F_1^R$. The computation of the homology with the constant cosheaf $\Z^\Sigma$ given in \cref{ex:cross_Z_BM_homology} carries through to $R^\Sigma$, giving $H_1^{BM}(\Sigma,\F_0^R)=R^3$ and $H_0^{BM}(\Sigma,\F_0^R)=0$. Finally, to compute the Borel--Moore homology for $\F_1^R$, we have the chain complex
	\begin{equation*}
		\begin{tikzcd}
			0 \arrow[r] &\oplus_{\tau_i \in \Sigma^1} \F_1^R(\tau_i) \arrow[r, "\partial_1"] & \F_1^R(v) \arrow[r] &0.
		\end{tikzcd}
	\end{equation*}
	Selecting the $\Z$-basis $e_1=(1,0)$, $e_2=(0,1)$ for $N$, we can write this complex as
	\begin{equation*}
		\begin{tikzcd}
			0 \arrow[r] &\langle(1,0)\rangle\oplus\langle(0,1)\rangle\oplus\langle(-1,0)\rangle\oplus\langle(0,-1)\rangle \arrow[r, "\partial_1"] &  \langle(1,0), (0,1)\rangle \arrow[r] &0,
		\end{tikzcd}
	\end{equation*}
	where $\partial_1$ is now the direct sum of the inclusion maps, and everything is suitably tensored with $R$. The Borel--Moore homology can then be shown to be given by 
	$H_0^{BM}(\Sigma,\F_1^R)=0$ and $H_1^{BM}(\Sigma,\F_1^R)=\langle (a,b,a,b)\mid a,b\in R \rangle \cong R^2$.
\end{example}
\begin{example}
	We now show how to perform the above computations using the \cite{CellularSheaves} package for \cite{polymake}, when working with rational coefficients. A code example is given in \cref{fig:cross_polymake}. To compute with \cite{polymake}, one specifies the rays of a fan, as well as which rays form a cone. The fan must be input in projective coordinates, so that there is a distinct projection point $[1,0,0]$, with all rays expressed using an embedding of $N$ into the hyperplane $H = \set{(x_0,x_1,x_2) \mid x_0=0}$. Thus the ray $\tau_1$ is $[0,1,0]$. Similarly the cones must all be given as including the projection point $[1,0,0]$, so that the one-dimensional ray $\tau_2$ is given as $[0,2]$.
	\begin{figure}[H]
		\centering
\begin{lstlisting}[language=Perl]
    application 'fan';
    $fan = new fan::PolyhedralFan(
        INPUT_RAYS=>[
            [1,0,0],[0,1,0],[0,0,1],[0,-1,0],[0,0,-1]
        ], 
        INPUT_CONES =>[
            [0,1],[0,2],[0,3],[0,4],
        ]
    );
    $complex = new fan::PolyhedralComplex($fan);
    $dim = $complex -> AMBIENT_DIM;

    @cohom_dimensions = ();
    @BM_dimensions = ();
    for(my $i=0;$i<$dim;$i++){
        my $fi = $complex->fcosheaf($i);
        my $si=$complex->usual_chain_complex($fi);
        my $bmi=$complex->borel_moore_complex($fi);
        push @cohom_dimensions, topaz::betti_numbers($si);
        push @BM_dimensions, topaz::betti_numbers($bmi);
    }

    print "H^q(Sigma, F^p) dimensions; q=columns, p=rows";
    print join("\n", @cohom_dimensions),"\n\n";
    print "H_q^{BM}(Sigma, F_p) dimensions; q=columns, p=rows";
    print join("\n", @BM_dimensions),"\n\n";
\end{lstlisting}
		\caption{Code in \cite{polymake} to compute the tropical cohomology and Borel--Moore homology in \cref{ex:cross}}
		\label{fig:cross_polymake}
	\end{figure}
	
	Note also that one may use the command \lstinline{$complex -> VISUAL;} to receive a visualisation for two- and three-dimensional fans. The output of the code in \cref{fig:cross_polymake} is shown in \cref{fig:cross_polymake_output}:
	\begin{figure}[H]
		\centering
\begin{lstlisting}
    H^q(Sigma, F^p) dimensions; q=columns, p=rows
    1 0
    2 0

    H_q^{BM}(Sigma, F_p) dimensions; q=columns, p=rows
    0 3
    0 2
\end{lstlisting}
		\caption{Output from \cref{fig:cross_polymake}}
		\label{fig:cross_polymake_output}
	\end{figure}
\end{example}

Recall from \cref{def:stars_cones} that one must subdivide the stars $\Star{\gamma}$ of faces $\gamma\in \Sigma$ to obtain a fan structure. The next proposition shows that the tropical cohomology of $\Star{\gamma}$ is determined directly by $\F_p^R (\gamma)$, and that the tropical Borel--Moore homology can be computed using a simpler complex than the one coming from the subdivision.
\begin{proposition}\label{prop:star_homology_cohomology}
	Let $\Sigma$ be a fan and $\gamma \in \Sigma$ a face of dimension $r$. Let $\F_{R,\Sigma}^p$ and $\F_{R,\gamma}^p$ denote the $p$-th multi-tangent sheaves on $\Sigma$ and $\Star{\gamma}$ respectively. Then 
	$$H^0 (\Star{\gamma}, \F_{R,\gamma}^p)\cong \F_{R,\Sigma}^p (\gamma).$$ 
	Similarly, let $\F_p^{R,\Sigma}$ and $\F_p^{R,\gamma}$ denote the $p$-th multi-tangent cosheaves on $\Sigma$ and $\Star{\gamma}$ respectively.
	Then the Borel--Moore homology $H_q^{BM}(\Star{\gamma}, \F_p^{R,\gamma})$ is isomorphic to the homology of the complex
	\[\begin{tikzcd}
		0 & {\bigoplus\limits_{\substack{\alpha \in \Sigma^d \\ \alpha \succ \gamma}} \F_p^{R,\Sigma} (\alpha)} & \cdots & {\bigoplus\limits_{\substack{\kappa \in \Sigma^{r+1} \\ \kappa \succ \gamma}} \F_p^{R,\Sigma} (\kappa)} & {\F_p^{R,\Sigma}(\gamma)} & 0,
		\arrow[from=1-3, to=1-4]
		\arrow["{\partial_q|_{\kappa \succ \gamma}}", from=1-4, to=1-5]
		\arrow[from=1-5, to=1-6]
		\arrow[from=1-1, to=1-2]
		\arrow["{\partial_q|_{\alpha \succ \gamma}}", from=1-2, to=1-3]
	\end{tikzcd}\]
	where we define ${\partial_{q}^\gamma}=\oplus \partial_{\sigma \tau}$ with the sum taken over all $\sigma, \tau \succeq \gamma$, $\sigma \in \Sigma^{q}$ and $\tau \in \Sigma^{q-1}$. 
\end{proposition}
\begin{proof}
	First, by \cref{def:stars_cones}, we must choose a subdivision of the space with support given by the cones $\widetilde{\kappa}=\set{t (x-y) \mid t \geq 0, x \in \kappa, y \in \gamma}$
	for each $\kappa \succeq \gamma$, and we will have only one compact cell given by the created vertex $\widetilde{v} \in \Star{\gamma}$.
	By \cref{remark:trivial_hom_cohom}, we have $H^0 (\Star{\gamma}, \F_{R,\gamma}^p)= \F_{R,\gamma}^p(\widetilde{v})$. Then, observe that for each $\kappa \succ \gamma$, the lattice is unchanged in the sense that $L_\Z(\widetilde{\kappa})= L_\Z(\kappa)$ as subspaces of $N$, and each maximal dimensional face $\alpha$ in the subdivision of $\Star{\gamma}$ is a subspace of a $\widetilde{\kappa}$. In particular, this implies that $\F_{R,\gamma}^p (\widetilde{v})\cong \F_{R,\Sigma}^p (\gamma)$.

	Next, observe that the Borel--Moore homology is the equal to the regular homology of the fan when seen as a subset of the one-point compactification $N\cup \set{\infty}$ of the ambient lattice $N$. Then every cone $\sigma \in \Star{\gamma}$ becomes a disk $\sigma_\infty$ in $N \cup \set{\infty}$, and we have a CW complex structure on the underlying set of $\Star{\gamma} \cup \set{\infty}$.
	Then, similarly to \cite[Section 2.2]{MZ14} and \cite[Remark 2.8]{Lefschetz11}, note that:
	$$C_q^{BM}(\Star{\gamma}, \F_p^R)=\oplus_{\sigma \in \Star{\gamma}^q} \F_p^R (\sigma)= \oplus_{\sigma \in \Star{\gamma}^q} H_q (\sigma_\infty, \partial (\sigma_\infty), \F_p^R(\sigma)),$$
	where the right hand side becomes the cellular homology with coefficients in the local system induced by $\F_p^R$ of the CW complex $\Star{\gamma}\cup \set{\infty}$. This can be computed with an arbitrary CW structure. Therefore, the given complex, which is the local system homology of the CW structure induced on $|\Star{\gamma} \cup \set{\infty}|$ by taking the non-subdivided cell structure of $\Star{\gamma}$ will compute the Borel--Moore homology of the $\F_p^R$ cosheaf.
\end{proof}
The above proposition shows that, when working with stars of faces in a fan, the particular cellular structure does not change the tropical (co)homology. This is not the case for general (co)sheaves, see for instance the \emph{wave tangent sheaves} defined in \cite[Section 3]{MZ14}. 

\subsection{Balancing in tropical geometry}

It was observed in \cite[Remark 4.9]{Lefschetz11} and \cite[Proposition 4.3]{MZ14} that the balancing condition from \cref{def:balancing_Z} can be equivalently formulated as the condition that a particular tropical Borel--Moore chain is closed. In this subsection, we use this observation to form a generalization of the balancing condition to arbitrary commutative rings.
\begin{definition}\label{def:weights_R}
	Let $R$ be a ring. An \emph{$R$-weight function} $w \colon \Sigma^d \to R$ on a $d$-dimensional fan $\Sigma$ is a function such that for all $\alpha \in \Sigma^d$, the weight $w(\alpha)$ is not a zero-divisor. A pair $(\Sigma,w)$ will be called an \emph{$R$-weighted fan}.
\end{definition}
\begin{definition}\label{def:balanced}
	An $R$-weighted fan $(\Sigma, w)$ is \emph{$R$-balanced} if \emph{the fundamental chain $\mathrm{Ch}(\Sigma,w)$} given by
	$$\mathrm{Ch}(\Sigma,w)\coloneqq(w(\alpha)\Lambda_\alpha)_{\alpha \in \Sigma^d} \in C_{d}^{BM}(\Sigma, \F_d)$$
	is a cycle, where the $\Lambda_\alpha$ are chosen as in \cref{remark:Lambda_generator}. In this case, we have $H_d^{BM}(\Sigma, \F_d^R) \not = 0$, together with an induced \emph{fundamental class} 
	$$[\Sigma,w] = [\mathrm{Ch}(\Sigma,w)] \in H_d^{BM}(\Sigma, \F_d^R).$$
	If $H_d^{BM}(\Sigma, \F_d^R) = \langle [\Sigma,w] \rangle \cong R$, we say that $\Sigma$ is \emph{uniquely $R$-balanced} by $w$.
\end{definition}
\begin{example}
	We now compute the fundamental chain in the fan $\Sigma$ of \cref{ex:complete}. Choose orientations such that the elements $\Lambda_{\sigma_i}$ are as in \cref{ex:complete_lambda}. Moreover, we choose a weight function assigning the value $1$ to each of the cones $\sigma_i$. Then the fundamental chain is:
	$$(\Lambda_{\sigma_i})_{i=1}^3 = (e_1 \wedge e_2, e_1 \wedge e_2,e_1 \wedge e_2) \in C_2^{BM}(\Sigma,\F_2^\Z).$$
	It is then straightforward to check that, under the boundary map $\partial_2$, taking into account orientations, this chain is mapped to zero. For instance, for the component of $C_1^{BM}(\Sigma,\F_2^\Z)$ corresponding to $\tau_1$, we have
	$$- e_1 \wedge e_2 + e_1 \wedge e_2 = 0 \in \F_2^\Z (\tau_1),$$
	with the first $2$-wedge corresponding to $\sigma_1$ and the second to $\sigma_2$.
	Thus the fan is $\Z$-balanced, and there is a fundamental class $[\Sigma,w]\in H_2^{BM}(\Sigma,\F_2^R)$. Moreover, it can be checked that this class generated the whole cohomology module, so that this fan is in fact uniquely $R$-balanced.
\end{example}
The above definition, which is equivalent to the usual balancing condition \cite[Definition 5.8]{BIMS14}, is similar in flavor to the description given by \cite[Theorem 2.9]{BabaeeHuh}. We illustrate this with the following example.
\begin{example}\label{ex:complete_fund_class}
	In this example, we explicitly relate the above definition of balancing to the one given in \cref{def:balancing_Z}. Let $(\Sigma,w)$ be a $\Z$-balanced fan of dimension $d$ in the sense of \cref{def:balanced}. Then, for each $\beta\in \Sigma^{d-1}$, we pick a generator $\Lambda_\beta \in L_\Z (\beta)$ respecting the orientation. Now for each $\alpha\in \Sigma^d$ with $\beta \prec \alpha$, the vector $v_{\alpha/\beta}$ from \cref{def:cone} is such that $\Lambda_\alpha = \Lambda_\beta \wedge v_{\alpha/\beta}$.

	Looking at the $\beta$-component of $\partial\colon C_d^{BM}(\Sigma, \F_d^\Z) \to  C_{d-1}^{BM}(\Sigma, \F_d^\Z)$, we have
	$$\partial((w(\alpha)\Lambda_\alpha))_\beta = \sum_{\beta \prec \alpha} w(\alpha)\Lambda_\alpha= \Lambda_\beta \wedge \sum_{\beta \prec \alpha} w(\alpha)v_{\alpha/\beta}.$$
	Therefore, the chain $\mathrm{Ch}(\Sigma,w)$ is closed if and only if each of the faces $\beta$ are balanced in the sense of \cref{def:balancing_Z}. Thus the definitions are equivalent.
\end{example}
\begin{proposition}\label{prop:stars_are_balanced}
	Let $(\Sigma,w)$ be an $R$-balanced fan and $\gamma\in \Sigma$ a face. Then $(\Star{\gamma},w)$ is $R$-balanced, where $w$ is understood to be the weight function induced on $\Star{\gamma}$ by $w$.
\end{proposition}
\begin{proof}
	By \cref{prop:star_homology_cohomology}, we have that $H_d^{BM}(\Star{\gamma}, \F_d^R)$ can be viewed as the kernel of the map
	$${\bigoplus_{\substack{\alpha \in \Sigma^d \\ \alpha \succ \gamma}} \F_d^R (\alpha)} \xrightarrow{\partial_d^\gamma} {\bigoplus_{\substack{\beta \in \Sigma^{d-1} \\ \beta \succ \gamma}} \F_d^R (\beta)}.$$
	Moreover, the class
	$$\mathrm{Ch}(\Star{\gamma}, w) = (w(\alpha) \Lambda_\alpha)_{\substack{\alpha \in \Sigma^d \\ \alpha \succ \gamma}}$$
	is a cycle since $\mathrm{Ch}(\Sigma,w)$ is a cycle in $C_d^{BM}(\Sigma,\F_d^R)$. Thus $(\Star{\gamma},w)$ is $R$-balanced and we have a fundamental class $[\Star{\gamma},w] \in H_d^{BM}(\Star{\gamma},\F_d^R)$.
\end{proof}

\subsection{Tropical cap product}
There is a cap product relating $H^q(\Sigma,\F_R^p)$ to $H_{d-q}^{BM}(\Sigma,\F_{d-p}^R)$, which will be at the core of tropical Poincaré duality. We extend the version given in \cite[Definition 4.10]{Lefschetz11} for $R= \Z$ to arbitrary commutative rings $R$, using the contraction map from multilinear algebra for a general ring $R$, as developed in \cite[Section III.11]{Bourbaki}.
\begin{definition}[{\cite[Section III.11.9]{Bourbaki}}]
	Let $M_R$ be a rank $d$ free $R$-module, $M_R^*$ the dual module, and $0\leq p_1 \leq p_2 \leq d$. The \emph{interior product} or \emph{contraction} defined by $y=y_1 \wedge \dots \wedge y_{p_2} \in \bigwedge^{p_2} M_R$ is the map
	\begin{align*}
		\lrcorner \, y \colon \bigwedge^{p_1} M_R^* &\to \bigwedge^{p_2-p_1} M_R,
	\end{align*}
	which is defined on $x=x_1\wedge \dots \wedge x_p \in \bigwedge^{p_1} M_R^*$ to be
	\begin{equation*}
		x \, \lrcorner \, y = (-1)^{p_1(p_1-1)/2} \sum_a \mathrm{sign}(a) \left(\prod_{i=1}^{p_1} x_i (y_{a(i)}) \right) y_{a(p+1)} \wedge \dots \wedge y_{a(d)},
	\end{equation*}
	where the sum is taken over all permutations $a\in S_d$ which are increasing on $1,\dots,p$ and $p+1,\dots, d$, and is extended linearly.
\end{definition}
\begin{remark}\label{remark:contraction_basis_formula}
	In \cite[Section III.11.10]{Bourbaki}, an explicit formula for this contraction map is given in terms of bases.  Letting $e_1, \dots, e_m$ be a basis of $M_R$, the elements $e_I\coloneqq e_{i_1}\wedge \dots \wedge e_{i_{p_2}} \in \bigwedge^{p_2} M_R$, for all $I=\set{i_1<\dots < i_{p_2}}\subseteq [m]$ ordered strictly increasing subsets of size $p$, form a basis of $\bigwedge^{p_2} M_R$. Letting $f_1,\dots, f_m$ be the dual basis to the $ e_i$ for $M_R^*$, the elements $f_J$, for all $J=\set{j_1<\dots < j_{p_1}}\subseteq [m]$, form a basis of $\bigwedge^{p_1} M_R^*$. Then the contraction map defined by $e_J$ is given by 
    $$\begin{cases*}
        f_K \, \lrcorner \, e_J = 0 & if $K \not\subset J$, \\
        f_K \, \lrcorner \, e_J = (-1)^{v + p_1(p_1-1)/2} \; e_{J \setminus K} & if $K\subseteq J$ and $p_1 = |K|$,
    \end{cases*}$$
    where $v$ is the number of ordered pairs $(\lambda, \mu) \in K \times (J \setminus K)$ such that $\lambda > \mu$.
\end{remark}
A proof that these contraction maps are the same as the formulation in terms of compositions given in \cite{AminiPiquerezFans,mastersthesis} follows from the arguments given in \cite[Lemma 4.1.4]{mastersthesis}.
\begin{definition}\label{def:contraction_on_face}
	For each facet $\alpha\in \Sigma^d$ of a $d$-dimensional $R$-balanced fan $\Sigma$, we have chosen a generator $\Lambda_\alpha \in \F_d^R (\alpha)= \bigwedge^d L_R (\alpha) \cong R$ by \cref{remark:Lambda_generator}, and a weight $w(\alpha) \in R$ which is not a zero-divisor. The \emph{contraction} defined by $w(\alpha) \Lambda_\alpha$ is the map 
	$$\lrcorner \, w(\alpha) \Lambda_\alpha \colon \bigwedge^{p} L_R(\alpha)^* \to \bigwedge^{d-p}L_R(\alpha).$$
	Since $w(\alpha)$ is not a zero-divisor, \cref{remark:contraction_basis_formula} shows that this map is injective. It is an isomorphism if and only if $w(\alpha)$ is a unit.
\end{definition}
\begin{definition}\label{def:cap_product}
	Let the weighted fan $(\Sigma,w)$ be an $R$-balanced fan of dimension $d$. The \emph{cap product $\frown \mathrm{Ch}(\Sigma,w)$} with the fundamental chain of $\Sigma$ is the map given by:
	\begin{align*}
		\frown\mathrm{Ch}(\Sigma,w)\colon C^q(\Sigma,\F_R^p) &\to C_{d-q}^{BM}(\Sigma,\F_{d-p}^R) \\
						(u_\gamma)_{\gamma \in \Sigma^q} &\mapsto \left(\sum_{\substack{\alpha \in \Sigma^d \\ \gamma, \tau  \preceq \alpha}} w(\alpha) \iota_{\alpha,\tau} \left( \rho_{\gamma,\alpha} (u_\gamma) \, \lrcorner \, \Lambda_\alpha \right) \right)_{\tau \in \Sigma^{d-q}}
	\end{align*}
	where $\Lambda_\alpha$ is as in \cref{remark:Lambda_generator}.
\end{definition}
\begin{remark}
	For any chain $\sigma\in C_q^{BM}(\Sigma,\F_{p}^R)$, a cap product $\frown\sigma$ can be similarly defined. It is noted in \cite[p. 13]{Lefschetz11} that the Leibniz formula holds for these cap product, such that $\partial (\alpha \frown\sigma) = (-1)^{q+1} (d(\alpha)\frown\sigma - \alpha \frown\partial(\sigma))$. In the case where $R=\R$, the Leibniz formula also follows from \cite[Remark 2.2, Definition 4.11]{JellShawSmacka}. Therefore the above defined map descends to tropical (co)homology, and we have the \emph{cap product with the fundamental class $\frown[\Sigma,w]$}
	\begin{align*}
		\frown[\Sigma,w] \colon H^q(\Sigma,\F_R^p) &\to H_{d-q}^{BM}(\Sigma,\F_{d-p}^R).
	\end{align*}
\end{remark}
\begin{example}\label{ex:complete_cap_product}
	In \cref{ex:complete_fund_class}, we computed the fundamental class of the fan $\Sigma$ from \cref{ex:complete}, given the all-one weight function $w$. We will now compute some examples of the cap product. Let $e_1,e_2\in N \cong \Z^2$ be the standard basis for the underlying lattice, and $e_1^*,e_2^*$ the dual basis for the dual lattice $N^*$. Then
	\begin{align*}
		H^0(\Sigma,\F_\Z^0) &= \F_\Z^0(v) = \Z, \\
		H^0(\Sigma,\F_\Z^1) &= \F_\Z^1(v) = \langle e_1^*, e_2^* \rangle_\Z, \\
		H^0(\Sigma,\F_\Z^2) &= \F_\Z^2(v) = \langle e_1^* \wedge e_2^* \rangle_\Z.
	\end{align*}
	Moreover, all other cohomology groups are zero, by \cref{prop:fp_cohom_trivial}. Next, the Borel--Moore chain complexes are given by 
	\begin{align*}
		C_\bullet^{BM}(\Sigma, \F_0^\Z) &\colon \quad 0 \to \Z^3 \to \Z^3 \to \Z \to 0,  \\
		C_\bullet^{BM}(\Sigma, \F_1^\Z) &\colon \quad 0 \to \bigoplus_{i=1}^3 \F_1^\Z (\sigma_i) \to \bigoplus_{j=1}^3 \F_1^\Z (\tau_j) \to \F_1^\Z (v) \to 0, \\
		C_\bullet^{BM}(\Sigma, \F_2^\Z) &\colon \quad 0 \to \bigoplus_{i=1}^3 \F_2^\Z (\sigma_i) \to \bigoplus_{j=1}^3 \F_2^\Z (\tau_j) \to \F_2^\Z (v) \to 0.
	\end{align*}
	We now compute $H_2^{BM}(\Sigma,\F_1^\Z)$ as an example.
	We have that $\F_1^\Z (\tau_j) = \langle e_1, e_2 \rangle_\Z$ for each $j$, and $\F_1^\Z (\sigma_i) = \langle e_1, e_2 \rangle_\Z$ for each $i$. Taking the direct sum of these bases in $\bigoplus_{i=1}^3 \F_1^\Z (\sigma_i)$, and respecting the orientations, we may express the differential $\partial_2: \bigoplus_{i=1}^3 \F_1^\Z (\sigma_i) \to \bigoplus_{j=1}^3 \F_1^\Z (\tau_j)$ in coordinates as:
	\begin{equation*}
		(\alpha^1,\beta^1,\alpha^2,\beta^2,\alpha^3,\beta^3) 
		\mapsto 
		(-\alpha^1+\alpha^2,-\beta^1+\beta^2,\alpha^1-\alpha^3,\beta^1-\beta^3,-\alpha^2+\alpha^3,-\beta^2+\beta^3),
	\end{equation*}
	with $\alpha^i$ and $\beta^i$ corresponding respectively to $e_1$ and $e_2$ for $\sigma_i$. We compute a basis for the kernel of this map, i.e. a basis for $H_2^{BM}(\Sigma,\F_1^\Z)$ to be:
	$$H_2^{BM}(\Sigma,\F_1^\Z) = \langle (1,0,1,0,1,0), (0,1,0,1,0,1) \rangle_\Z \subset C_2^{BM}(\Sigma, \F_1^\Z)$$
	Similar computations show that the remaining Borel--Moore homology groups are given by:
	\begin{equation*}
		\begin{matrix}
			H_0^{BM}(\Sigma, \F_0^\Z)=0  & H_1^{BM}(\Sigma, \F_0^\Z)=0  & H_2^{BM}(\Sigma, \F_0^\Z) \cong \Z \\
			H_0^{BM}(\Sigma, \F_1^\Z)=0  & H_1^{BM}(\Sigma, \F_1^\Z)=0  & H_2^{BM}(\Sigma, \F_1^\Z) \cong \Z^2 \\
			H_0^{BM}(\Sigma, \F_1^\Z)=0  & H_1^{BM}(\Sigma, \F_1^\Z)=0  & H_2^{BM}(\Sigma, \F_2^\Z) = \langle [\Sigma,w] \rangle \cong \Z.
		\end{matrix}
	\end{equation*}
	Finally, we now compute an example for the cap product map, in particular $\frown [\Sigma,w] \colon H^0(\Sigma,\F_\Z^1) \to H_2^{BM}(\Sigma,\F_1^\Z)$. Working from the definition, we have that
	$$e_1^* \mapsto (e_1^* \lrcorner \Lambda_{\sigma_1},e_1^* \lrcorner \Lambda_{\sigma_2},e_1^* \lrcorner \Lambda_{\sigma_3})\in \bigoplus_{i=1}^3 \F_1^\Z (\sigma_i).$$
	Expanding this using the $(\alpha^i,\beta^i)$ basis from above, these contractions are such that $e_1^* \mapsto (1,0,1,0,1,0)$, and one can similarly check that $e_2^* \mapsto (0,1,0,1,0,1)$. This shows that $\frown [\Sigma,w]$ is in this case an isomorphim.
\end{example}

\begin{proposition}\label{prop:cap_zero_map}
	Let $(\Sigma,w)$ be an $R$-balanced fan of dimension $d$. The cap product with the fundamental class $\frown[\Sigma,w]$ in tropical cohomology
	\begin{align*}
		\frown[\Sigma,w] \colon H^q(\Sigma,\F_R^p) &\to H_{d-q}^{BM}(\Sigma,\F_{d-p}^R)
	\end{align*}
	is the $0$-map for $q\not = 0$.
\end{proposition}
\begin{proof}
	By \cref{prop:fp_cohom_trivial}, we have that $H^q (\Sigma, \F_p^R)=0$ for $q\neq 0$, hence this cap product is only non-trivial when $q\not = 0$,
\end{proof}
The above proposition shows that in the fan-case, the only interesting cap products are of the form $\frown[\Sigma,w] \colon H^0(\Sigma,\F_R^p) \to H_{d}^{BM}(\Sigma,\F_{d-p}^R)$, for $p=0,\dots, d$. Moreover, in \cref{prop:cap_injective} below, we show that these are injective for any commutative ring $R$. In the case where $R=\R$, this was shown in \cite[Theorem 4.3.1]{mastersthesis}, and for $R=\Z$, it is stated in \cite[Section 3.2.2]{AminiPiquerezFans}.
\begin{proposition}\label{prop:cap_injective}
	For an $R$-balanced fan $(\Sigma,w)$ of dimension $d$, the map
	\begin{align*}
		\frown[\Sigma,w] \colon H^0(\Sigma,\F_R^p) &\to H_{d}^{BM}(\Sigma,\F_{d-p}^R)
	\end{align*}
	is injective.
\end{proposition}
\begin{proof}
	We have that $H_{d}^{BM}(\Sigma,\F_{d-p}^R) = \ker(\partial_d)$, and $H^0(\Sigma,\F_R^p) = \F_R^p (v)$, so that $\frown[\Sigma,w]$ is exactly 
	\begin{align*}
		\frown\mathrm{Ch}(\Sigma,w) \colon \F_R^p (v) & \to \bigoplus_{\alpha \in \Sigma^d} \F_{d-p}^R (\alpha) \\
		u &\mapsto (   \rho_{v,\alpha} (u) \, \lrcorner \, w(\alpha)\Lambda_\alpha )_{\alpha \in \Sigma^d}
	\end{align*}
	where the image lies in $H_{d}^{BM}(\Sigma,\F_{d-p}^R) \subseteq \bigoplus_{\alpha \in \Sigma^d} \F_{d-p}^R (\alpha)$.
	This is the composition of the map $\oplus_\alpha \rho_{v,\alpha} \colon \F_R^p (v) \to \oplus_{\alpha \in \Sigma^d} \F_R^p (\alpha)$, which is injective, since it is dual to the surjection $\oplus_{\alpha \in \Sigma^d} \F_p^R (\alpha) \to  \F_p^R(v)$, and the direct sum of the contractions $\oplus_{\alpha \in \Sigma^d} \, \lrcorner\, w(\alpha)\Lambda_\alpha$, which are injective (\cref{def:contraction_on_face}). Thus this cap product is the composition of injective maps and is therefore injective.
\end{proof}
\begin{proposition}\label{prop:cap_product_on_stars}
	Let $(\Sigma,w)$ be an $R$-balanced fan and $\gamma\in \Sigma$ a face. Then the cap product map on the star fan
	$$\frown [\Star{\gamma},w]\colon H^0(\Star{\gamma}, \F_R^p) \to H_d^{BM}(\Star{\gamma}, \F_{d-p}^{R})$$
	is given by
 	$$u \mapsto (  u \,\lrcorner\, w(\alpha)\Lambda_\alpha )_{\substack{\alpha \in \Sigma^d \\ \alpha \succ \gamma}},$$
	where we identify $H^0(\Star{\gamma}, \F_R^p)\cong \F_p^R(\gamma)$, and $H_d^{BM}(\Star{\gamma}, \F_{d-p}^{R})$ as the kernel of the first map in the complex from \cref{prop:star_homology_cohomology}.
\end{proposition}
\begin{proof}
	The identifications are justified by \cref{prop:star_homology_cohomology}, and the existence of this fundamental class by \cref{prop:stars_are_balanced}. It remains to show that the stated formula corresponds to the cap product. 
	
	Consider a subdivision making $\Star{\gamma}$ a pointed fan. Each $d$-cell $\widetilde{\alpha}$ of the subdivision maps to a $d$-cell $\alpha\succ \gamma$ of $\Sigma$, similarly to the proof of \cref{prop:star_homology_cohomology}. The formula then follows from the induced map in homology.
\end{proof}

%% file: sections/duality.tex
\section{Tropical Poincaré duality}\label{section:Tropical_Poincare_duality}
In \cref{subsection:tpd_def_ex}, we define TPD over a ring $R$, and give an example of a non-matroidal fan satisfying the duality. In \cref{subsection:nec_cond_euler_char}, we give some necessary conditions for the duality to hold, along with a characterization by an Euler characteristic condition. Finally, in \cref{subsection:dim_codim1}, we turn to the problem of determining which fans are TPD spaces. We classify all the one-dimensional fans satisfying TPD over a ring $R$ and study tropical fan hypersurfaces in $\R^n$ satisfying TPD. This forms a first step towards answering \cref{question:tpd_when}.

\subsection{Definition and examples}\label{subsection:tpd_def_ex}
In this subsection, we define what it means for a fan to satisfy TPD over a commutative ring $R$. When $R=\Z$, this is the definition from \cite[Definition 5.2]{Lefschetz11}, and when $R=\R$, our definition can be shown to be equivalent to \cite[Definition 4.12]{JellShawSmacka}.

\begin{definition}\label{def:tpd}
	We say that an $R$-balanced rational polyhedral fan $\Sigma$ of dimension $d$ with weights $w$ satisfies \emph{tropical Poincaré duality over $R$} if the cap product with the fundamental class 
	$$\frown[\Sigma,w] \colon H^q (\Sigma,\F_R^{p})\to H_{d-q}^{BM}(\Sigma, \F_{d-p}^R)$$
	is an isomorphism for all $p,q=0,\dots,d$.
\end{definition}
\begin{example}
	Returning again to \cref{ex:complete}, one can verify that all the possible cap products are isomorphisms, as we did explicitly in \cref{ex:complete_cap_product} for the cap product $\frown[\Sigma,w] \colon H^0 (\Sigma,\F_\Z^{1})\to H_{2}^{BM}(\Sigma, \F_{1}^\Z)$, so that this fan satisfies tropical Poincaré duality over $\Z$.
\end{example}
\begin{example}
	Similarly, explicit computations can be carried out for \cref{ex:cross}. Comparing back to \cref{ex:cross_homology}, we have that $\dim_\Z H_1^{BM}(\Sigma,\F_1^R)=2$ and $\dim_\Z H_1^{BM}(\Sigma,\F_0^\Z) = 3$, while $\dim_\Z H^0(\Sigma,\F_1^R)=2$ and $\dim_\Z H^0(\Sigma,\F_0^R)=1$. 
	
	Thus the cap product maps 
	\begin{align*}
		\frown[\Sigma,w] \colon H^0 (\Sigma,\F_\Z^{0}) &\to H_{1}^{BM}(\Sigma, \F_{1}^\Z), \quad \text{and} \\
		\frown[\Sigma,w] \colon H^0 (\Sigma,\F_\Z^{1}) &\to H_{1}^{BM}(\Sigma, \F_{0}^\Z)
	\end{align*}
	are not isomorphisms, and the fan does not satisfy tropical Poincaré duality over $\Z$.
\end{example}

As mentioned in the introduction, the Bergman fans of matroids satisfy TPD over $\R$ and $\Z$ \cite{JellShawSmacka, Lefschetz11}, however these are not the only such fans, as can be seen from the next example.
\begin{example} \label{ex:not_all_tpd_are_bergman_fans}
	\begin{figure}
		\centering
		\includestandalone[width=0.7\textwidth]{figures/babaee_huh_graph}
		\caption{The graph of cones for \cref{ex:not_all_tpd_are_bergman_fans}.}
		\label{fig:example2_graph}
	\end{figure}
	Let $f_1 \coloneqq (0,1,1,1)$, $f_2 \coloneqq (1,0,-1,1)$, $f_3 \coloneqq (1,1,0,-1)$ and $f_4 \coloneqq (1,-1,1,0)$ be vectors in $\R^4$ and let $e_1,e_2,e_3$ and $e_4$ be the standard basis. Consider the fan generated by the cones of vertices connected by an edge in \cref{fig:example2_graph}, so that for instance the cone of $e_1$ and $f_2$ is included. This fan was used in \cite{BabaeeHuh} to construct a counter-example to the strongly positive Hodge conjecture. It is not matroidal, since it does not satisfy the Hard Lefschetz property of \cite{AdiprasitoHuhKatz}.

	We compute its cellular tropical homology and cohomology over $\Q$ using the cellular sheaves package \cite{CellularSheaves} for \cite{polymake}, we have:
	\begin{equation*}
		\begin{aligned}[c]
			H^0(\Sigma, \F_\Q^0) &\cong \Q \\
			H^0(\Sigma, \F_\Q^1) &\cong \Q^4 \\
			H^0(\Sigma, \F_\Q^2) &\cong \Q^5
		\end{aligned}
		\qquad\text{ and }\qquad
		\begin{aligned}[c]
			H_2^{BM}(\Sigma, \F_2^\Q) &\cong \Q \\
			H_2^{BM}(\Sigma, \F_1^\Q) &\cong \Q^4 \\
			H_2^{BM}(\Sigma, \F_0^\Q) &\cong \Q^5 
		\end{aligned}
	\end{equation*}
	with all other groups being zero. By \cref{prop:cap_injective}, the cap product is injective, and since the dimensions agree, the cap products are isomorphisms when they are nonzero. Hence the fan satisfies TPD over $\Q$, where the weights for the fundamental class are chosen so as to form a generator of $H_2^{BM}(\Sigma, \F_2^\Q) = \Q$.
\end{example}

\subsection{Necessary conditions for tropical Poincaré duality}\label{subsection:nec_cond_euler_char}
We now turn to giving some necessary conditions for TPD to hold.

First, in light of \cref{prop:fp_cohom_trivial}, the Borel--Moore homology of fans satisfying TPD is concentrated in degree $d$. Indeed, by \cref{prop:fp_cohom_trivial}, $ H^q (\Sigma,\F_R^{p})=0$ for $q\not=0$, hence the isomorphism $\frown[\Sigma,w] \colon H^q (\Sigma,\F_R^{p})\cong H_{d-q}^{BM}(\Sigma, \F_{d-p}^R)$ gives $ H_{q}^{BM}(\Sigma, \F_{d-p}^R)=0$ for $q\not=d$.

Note also that, for $(\Sigma,w)$ be an $R$-balanced fan satisfying TPD over $R$, $\Sigma$ must be uniquely $R$-balanced by $w$. This is because the cap product maps $1\in R \cong H^0 (\Sigma, \F_R^0)$ (see \cref{prop:fp_cohom_trivial}) to $1\frown [\Sigma,w] = [\Sigma,w]\in H_d^{BM} (\Sigma,\F_d^R)$, which must be a generator. Then by \cref{def:balanced}, the fan $\Sigma$ is uniquely $R$-balanced.

\begin{example}
	For any ring $R$, the fan in \cref{fig:cross} is $R$-balanced, but not uniquely $R$-balanced by \cref{ex:cross_homology}, hence it cannot satisfy TPD over $R$.
\end{example}

Now, assuming that we are working over a field $\mathbbm{k}$, and that the Borel--Moore homology of the fan vanishes in an appropriate way, we can determine that the fan satisfies TPD through an Euler characteristic argument.
\begin{proposition}\label{prop:euler_char_condition}
	Let $\mathbbm{k}$ be a field, and $(\Sigma,w)$ be a $\mathbbm{k}$-balanced fan of dimension $d$. Suppose $H_q^{BM}(\Sigma,\F_p^\mathbbm{k})=0$ for $q\neq d$. Then, for a given $p$, the cap product 
	$$\frown[\Sigma,w] \colon H^q (\Sigma,\F_\mathbbm{k}^{p})\to H_{d-q}^{BM}(\Sigma, \F_{d-p}^\mathbbm{k})$$
	is an isomorphism for all $q$ if and only if
	\begin{equation}\label{eq:euler_char}
		(-1)^d \chi (C_\bullet^{BM}(\Sigma, \F_{d-p}^\mathbbm{k})) = \dim_k \F_\mathbbm{k}^p (v).
	\end{equation}
	Moreover, $(\Sigma,w)$ satisfies TPD over $\mathbbm{k}$ if and only if \cref{eq:euler_char} holds for all $p$.
\end{proposition}
\begin{proof}
	Since the only compact cell in $\Sigma$ is the vertex $v$, we have
	$$H^q(\Sigma, \F_\mathbbm{k}^p) = 
	\begin{cases*}
		\F_\mathbbm{k}^p (v) & if $q=0$, \\
		0 & otherwise.
	\end{cases*}$$
	By the vanishing condition on tropical Borel--Moore homology, the cap product $\frown[\Sigma,w] \colon H^q (\Sigma,\F_\mathbbm{k}^{p})\to H_{d-q}^{BM}(\Sigma, \F_{d-p}^\mathbbm{k})$
	is then immediately an isomorphism for $q\neq 0$, and 
	\begin{equation}\label{eq:euler_char_vanishing}
		\chi (C_\bullet^{BM}(\Sigma, \F_{d-p}^\mathbbm{k}))= \chi (H_\bullet^{BM}(\Sigma, \F_{d-p}^\mathbbm{k})) = (-1)^d \dim_\mathbbm{k} H_d^{BM}(\Sigma, \F_{d-p}^\mathbbm{k}).
	\end{equation}
	Since the cap product is injective by \cref{prop:cap_injective} and we are working over a field, the maps
	$$\frown[\Sigma,w] \colon H^0 (\Sigma,\F_\mathbbm{k}^{p})\to H_{d}^{BM}(\Sigma, \F_{d-p}^\mathbbm{k})$$
	are isomorphisms if and only if $\dim_\mathbbm{k} H^0 (\Sigma,\F_\mathbbm{k}^{p}) = \dim_\mathbbm{k} \F_\mathbbm{k}^p (v)$ is equal to $\dim_\mathbbm{k} H_d^{BM}(\Sigma, \F_{d-p}^\mathbbm{k})$. By \cref{eq:euler_char_vanishing}, this is exactly the claimed result.
\end{proof}

\subsection{Dimension one and codimension one}\label{subsection:dim_codim1}
We completely classify rational polyhedral fans of dimension 1 satisfying TPD over an arbitrary commutative ring. 
We begin with a utility lemma:
\begin{lemma}\label{lemma:bm_vanish_dim1}
	Let $R$ be a commutative ring, and $(\Sigma,w)$ an $R$-balanced fan of dimension one.
	Then we have $H_0^{BM}(\Sigma, \F_0^R)=0$ and $H_0^{BM}(\Sigma, \F_1^R)=0$.
\end{lemma}
\begin{proof}
	Let $v\in \Sigma$ be the vertex of the fan. By \cref{def:tropical_homology_cohomology}, the tropical Borel--Moore cochain complexes are:
	\begin{align*}
		C_\bullet^{BM}(\Sigma, \F_0^R) &\colon \quad \bigoplus_{\epsilon\in \Sigma^1} R \xrightarrow{\partial_1^0} R \to 0, \quad \text{and} \\
		C_\bullet^{BM}(\Sigma, \F_1^R) &\colon \quad \bigoplus_{\epsilon\in \Sigma^1} \F_1^R (\epsilon) \xrightarrow{\partial_1^1} \F_1^R (v) \to 0.
	\end{align*}
	Here $\partial_1^0$ is the map given by the matrix $(1 \; 1 \; \dots \; 1)$. It is surjective and thus $H_0^{BM}(\Sigma, \F_0^R)=0$.
	Similarly, $\F_1^R(v)=\sum_{\epsilon\in \Sigma^1} L_R (\epsilon) = \sum_{\epsilon\in \Sigma^1} \F_1^R (\epsilon)$, thus $\partial_1^1$ is surjective, hence $H_0^{BM}(\Sigma, \F_1^R)=0$.
\end{proof}

\begin{theorem}\label{thm:class_dim_1}
	Let $R$ be a commutative ring, and $(\Sigma,w)$ an $R$-balanced fan of dimension one.
	Then $(\Sigma,w)$ satisfies tropical Poincaré duality over $R$ if and only if it is uniquely $R$-balanced and all the weights are units in $R$.
\end{theorem}
\begin{proof}
	We need to show that all four of the following cap products
	\begin{enumerate}
		\item\label{thm:class_dim_1_item_1} $\frown [\Sigma,w] \colon H^1(\Sigma, \F_R^1) \to H_0^{BM}(\Sigma, \F_0^R)$,
		\item\label{thm:class_dim_1_item_2} $\frown [\Sigma,w] \colon H^1(\Sigma, \F_R^0) \to H_0^{BM}(\Sigma, \F_1^R)$,
		\item\label{thm:class_dim_1_item_3} $\frown [\Sigma,w] \colon H^0(\Sigma, \F_R^0) \to H_1^{BM}(\Sigma, \F_1^R)$,
		\item\label{thm:class_dim_1_item_4} $\frown [\Sigma,w] \colon H^0(\Sigma, \F_R^1) \to H_1^{BM}(\Sigma, \F_0^R)$,
	\end{enumerate}
	are isomorphisms if and only if $(\Sigma,w)$ is uniquely $R$-balanced and all the weights are units in $R$. We will show this in three parts:
	\begin{enumerate}[(a)]
		\item\label{part_a} First, we show that the maps \eqref{thm:class_dim_1_item_1} and \eqref{thm:class_dim_1_item_2} are trivial maps between zero-modules.
		\item\label{part_b} Then we show that \eqref{thm:class_dim_1_item_3} being an isomorphism is the definition of being uniquely $R$-balanced. 
		\item\label{part_c} Finally, we show that \eqref{thm:class_dim_1_item_4} is an isomorphism if and only if $(\Sigma,w)$ is uniquely $R$-balanced, with the added condition that all the weights are units in $R$.
	\end{enumerate}
	In total, this will then show that $(\Sigma,w)$ satisfies tropical Poincaré duality over $R$ if and only if it is uniquely $R$-balanced and all the weights are units in $R$.

	Beginning with \ref{part_a}, by \cref{lemma:bm_vanish_dim1} and \cref{prop:fp_cohom_trivial}, all involved modules are zero. Moreover the cap product map is zero by \cref{prop:cap_zero_map}, hence the maps \eqref{thm:class_dim_1_item_1} and \eqref{thm:class_dim_1_item_2} are trivially isomorphisms.

	Next for \ref{part_b}, the map $\frown [\Sigma,w] \colon H^0(\Sigma, \F_R^0) \to H_1^{BM}(\Sigma, \F_1^R)$ is given by sending a scalar $\alpha\in H^0(\Sigma, \F_R^0) \cong R$ to $\alpha \frown [\Sigma,w]$. The $0$-contraction of a scalar is multiplication by this scalar, so that $\alpha \frown [\Sigma,w]=\alpha \cdot [\Sigma,w]$. It is therefore an isomorphism if and only if $\langle[\Sigma,w]\rangle$ generates $H_1^{BM}(\Sigma, \F_1^R)$, which is the definition of uniquely $R$-balanced (\cref{def:balanced}).

	Finally, we turn to \ref{part_c}. We begin with some notation. Let $v$ be the vertex of $\Sigma$ and number the one-dimensional rays as $\epsilon_1, \dots, \epsilon_m$, with weights $w_i=w(\epsilon_i)$. 
	The Borel--Moore cochain group is $C_1^{BM}(\Sigma, \F_0^R) = \oplus_{i=1}^m R$, which has a basis $x_1,\dots,x_m$, with $x_i$ corresponding to $\epsilon_i$. The elements $x_i-x_m \in C_1^{BM}(\Sigma,\F_0^R)$, for $i=1,\dots,m-1$ form a basis for $H_1^{BM}(\Sigma,\F_0^R)= \ker(1 \; 1 \; \dots \; 1)$. 
	For each $\epsilon_i$, we select the generator $\Lambda_i\in L_R(\epsilon_i)$ compatible with the orientation of $\epsilon_i$, and let $\Theta_i \coloneqq \iota_{\epsilon_i,v}(\Lambda_i)$ be its image under the inclusion $\iota_{\epsilon_i,v}\colon \F_1^R(\epsilon_i) \to \F_1^R(v)$. 
	Thus the fundamental class $[\Sigma,w]\in H_1^{BM}(\Sigma,\F_R^1)$ is explicitly the element $(w_i \Lambda_i)_{i=1}^m \in H_1^{BM}(\Sigma,\F_1^R) = \ker((\iota_{\epsilon_i,v})_{i=1}^m)\subset C_1^{BM}(\Sigma,\F_1^R)$.
	The cap product map $\frown [\Sigma,w] \colon H^0(\Sigma, \F_R^1) \to H_1^{BM}(\Sigma, \F_0^R)$ takes a covector $\phi \in H^0(\Sigma, \F_R^1)= \F_R^1 (v)$ to the element
	$$(w_i \phi(\Theta_i))_{i=1}^m \in H_1^{BM}(\Sigma, \F_0^R).$$
	
	Now, suppose all the weights $w_i$ are units in $R$ and $(\Sigma,w)$ is uniquely $R$-balanced. Then the elements $w_i \Theta_i$, for $i=1,\dots, m-1$, form a basis for $\F_1^R(v)$, with the corresponding dual basis $w_i^{-1}\Theta_i^*$ for $\F_R^1(v)$. Then, for each $j=1,\dots,m-1$, 
	$$w_j^{-1}\Theta_j^* \frown [\Sigma,w] = (w_i w_j^{-1}\Theta_j^*(\Theta_i))_{i=1}^m = (0,\dots,0, 1,0, \dots, w_m w_j^{-1}\Theta_j^* (\Theta_m)),$$
	where the only two non-zero entries are in the $j$-th and $m$-th positions. Since this is a cycle in $C_1^{BM}(\Sigma,\F_0^R)$, we must have
	$1 + w_m w_j^{-1}\Theta_j^* (\Theta_m) =0,$
	so that 
	$$w_j^{-1}\Theta_j^* \frown [\Sigma,w]=x_i-x_m.$$
	Thus the images of the basis elements $w_j^{-1}\Theta_j^*$ of $\F_1^R(v)$ form a basis of $H_1^{BM}(\Sigma,\F_0^R)$, hence $\frown [\Sigma,w]$ is an isomorphism.
	
	For the converse direction, we show that if either the weights are non-units or the fan is not uniquely $R$-balanced, then the cap product is not an isomorphism.
	
	First, suppose some weight $w_k$ is not a unit in $R$. Then for any $\phi \in H^0(\Sigma, \F_R^1)= \F_R^1 (v)$, the $k$-th component of $\phi \frown [\Sigma,w]$ is contained in the ideal $\langle w_k \rangle \subset R$, which does not contain $1$. Hence the element $x_k-x_m$ of $H_1^{BM}(\Sigma,\F_0^R)$ cannot be in the image of $\frown [\Sigma,w]$, which is therefore not surjective and hence not an isomorphism.

	Finally, suppose that $\Sigma$ is not uniquely $R$-balanced. Since $H_1^{BM}(\Sigma,\F_0^R)$ is free of rank $m-1$, we may assume that $\F_R^1(v)$ is as well, otherwise there cannot be an isomorphism. Since $\Sigma$ is not uniquely $R$-balanced, $\rank_R H_1^{BM}(\Sigma,\F_1^R) > 1$, so that by working with the Euler characteristics, we must have $\rank_R \F_1^R(v) < m-1$. Dualizing, we obtain that $\rank_R \F_R^1(v) < m-1 = \rank_R H_1^{BM}(\Sigma,\F_0^R)$ and so the cap product cannot be an isomorphism. 
\end{proof}
\begin{corollary}\label{cor:dim_1_field}
	Let $\mathbbm{k}$ be a field, $(\Sigma,w)$ a $\mathbbm{k}$-balanced fan of dimension one. Then $(\Sigma,w)$ satisfies TPD over $\mathbbm{k}$ if and only if it is uniquely $\mathbbm{k}$-balanced.
\end{corollary}
\begin{proof}
	By \cref{thm:class_dim_1}, $(\Sigma,w)$ satisfies TPD if and only if it is uniquely $\mathbbm{k}$-balanced, and all the weights are units in $\mathbbm{k}$. The weights are non-zero by \cref{def:weights_R}, hence must be units since $\mathbbm{k}$ is a field.
\end{proof}
\begin{example}
	Let $\Sigma \subset \Z^3$ be the 1-dimensional fan with a vertex at the origin, and the four cones $\sigma_1,\sigma_2,\sigma_3$ and $\sigma_4$ generated by the vectors $\nu_1=(1,0,2), \, \nu_2=(-1,0,0), \, \nu_3=(0,-1,0), \, \nu_4=(0,1,-2)$ respectively. This is a balanced fan with the constant unit weight function $w(\sigma_i)=1$. The Borel--Moore chain complex $C_\bullet^{BM}(\Sigma,\F_1^\Z)$ can be written as
	\[\begin{tikzcd}
		0 & {\langle \Lambda_1 \rangle_\Z\oplus \langle \Lambda_2 \rangle_\Z \oplus \langle \Lambda_3 \rangle_\Z \oplus \langle \Lambda_4 \rangle_\Z} & {\F_1^\Z(v)} & 0
		\arrow[from=1-1, to=1-2]
		\arrow["{(\iota_{\sigma_i,v})}", from=1-2, to=1-3]
		\arrow[from=1-3, to=1-4]
	\end{tikzcd}\]
	Since $\iota_{\sigma_i,v}(\Lambda_i)=\nu_i$, we see in fact that $H_1^{BM}(\Sigma,\F_1^\Z)=\langle [\Sigma,w] \rangle$, where $[\Sigma,w]=(\Lambda_1,\Lambda_2,\Lambda_3,\Lambda_4)$. Thus $(\Sigma,w)$ is also uniquely $\Z$-balanced. 
	Moreover, the complex $C_\bullet^{BM}(\Sigma,\F_0^\Z)$ is
	\[\begin{tikzcd}
		0 & {\Z^4} & \Z & 0,
		\arrow["{(1\, 1\,1\,1)}", from=1-2, to=1-3]
		\arrow[from=1-3, to=1-4]
		\arrow[from=1-1, to=1-2]
	\end{tikzcd}\]
	so that $H_1^{BM}(\Sigma,\F_0^\Z)\cong \Z^3$. We pick the basis $\nu_1, \nu_2, \nu_3$ for $\F_\Z^1 (v)$, and balancing gives $\nu_4=-\nu_1-\nu_2-\nu_3$. The dual basis for $\F_\Z^1 (v)$ is $\nu_1^*, \nu_2^*, \nu_3^*$. We see that 
	\begin{align*}
		\nu_1^* \frown [\Sigma,w] &=(\oplus \lrcorner \, \Lambda_{\sigma_i})(\oplus \rho_{v,\sigma_i})(\nu_1^*) \\
		&=(\oplus \lrcorner \, \Lambda_{\sigma_i})( \oplus (\nu_1^* \circ \iota_{\sigma_i,v})) \\
		&=(\nu_1^* (\iota_{\sigma_1,v}(\Lambda_1) ),\; \nu_1^*(\iota_{\sigma_2,v}(\Lambda_2)),\;\nu_1^*(\iota_{\sigma_3,v}(\Lambda_3)),\; \nu_1^*(\iota_{\sigma_4,v}(\Lambda_4))) \\
		&= (\nu_1^* (\nu_1), \nu_1^* (\nu_2), \nu_1^* (\nu_3), \nu_1^* (\nu_4)) \\
		&= (1,0,0,-1).
	\end{align*}
	Similarly, $\nu_2^* \frown [\Sigma,w]=(0,1,0,-1)$, and $\nu_3^* \frown [\Sigma,w]=(0,0,1,-1)$. Since the images of the generating set $y_1^*,y_2^*,y_3^*$ for $\F_\Z^1(v)$ is a generating set for $H_1^{BM}(\Sigma,\F_0^\Z)\cong \Z^3$, the cap product is an isomorphism.
\end{example}

For codimension 1 fan tropical cycles in $\R^n$, which are fan tropical hypersurfaces,  we can characterize the Newton polytopes of the hypersurfaces having TPD. We refer to \cite[Chap. 2]{MikhalkinRauBook} for background on tropical hypersurfaces in $\R^n$, which they call \emph{very affine tropical hypersurfaces}.
\begin{proposition}\label{prop:fan_tropical_hypersurface}
	Let $f\in \T [x_0^{\pm 1},\dots , x_{d}^{\pm 1}]$ be a tropical Laurent polynomial such that the very affine tropical cycle $X=V(f) \subset \R^{d+1}$ is supported on a pointed fan. If $X$ satisfies TPD over a commutative ring $R$, then the Newton polytope $\Delta(f)$ of $f$ is a simplex.
\end{proposition}
\begin{proof}
	By assumption, the very affine tropical cycle $X$ is a pointed $d$-dimensional rational polyhedral fan (\cite[Cor 2.3.2]{MikhalkinRauBook}), thus $H^q(X,\F_R^p)=0$ for all $q>0$ and all $p$, and the isomorphisms $\frown [X] \colon H^q(X,\F_R^p) \to H_{d-q}^{BM}(X, \F_{d-p}^R)$ give in particular that $H_{d-q}^{BM}(X, \F_{0}^R)=0$ for all $q>0$.

	Since $X$ is the $d$-skeleton of the dual fan to $\Delta(f)$ by \cite[Thm 2.3.7, Cor 2.3.2]{MikhalkinRauBook}, 
	$\dim_R H_{d}^{BM} (X, \F_0 )$ is $\#\mathrm{Vert}(\Delta(f))-1$, the number of vertices of the polytope $\Delta(f)$, minus $1$.
	
	Since $X$ is $d$-dimensional, $\dim_R H^0 (X, \F_R^p)= \dim_R \F_R^p (v) = \binom{d+1}{p}$, thus by Poincaré duality we have $$\#\mathrm{Vert}(\Delta(f))-1 = \dim_R H_d^{BM}(X, \F_0^R) = \dim_R H^0 (X, \F_R^d)=\binom{d+1}{d}=d+1,$$ 
	and so $\Delta(f)$, being $(d+1)$-dimensional and having $d+2$ vertices, is a simplex.
\end{proof}

%% file: figures/babaee_huh_graph.tex
\begin{tikzpicture}[thick,scale=0.9, every node/.style={scale=0.9}]
\draw[fill=black] (0,0) circle (3pt);
\draw[fill=black] (2,0) circle (3pt);
\draw[fill=black] (4,0) circle (3pt);
\draw[fill=black] (6,0) circle (3pt);
\draw[fill=black] (8,0) circle (3pt);
\draw[fill=black] (10,0) circle (3pt);
\draw[fill=black] (12,0) circle (3pt);

\draw[fill=black] (0,4) circle (3pt);
\draw[fill=black] (2,4) circle (3pt);
\draw[fill=black] (4,4) circle (3pt);
\draw[fill=black] (6,4) circle (3pt);
\draw[fill=black] (8,4) circle (3pt);
\draw[fill=black] (10,4) circle (3pt);
\draw[fill=black] (12,4) circle (3pt);

\node at (0,-0.5) {$f_1$};
\node at (2,-0.5) {$f_2$};
\node at (4,-0.5) {$f_3$};
\node at (6,-0.5) {$f_4$};
\node at (8,-0.5) {$-f_4$};
\node at (10,-0.5) {$-f_2$};
\node at (12,-0.5) {$-f_3$};

\node at (0,4.5) {$e_1$};
\node at (2,4.5) {$e_2$};
\node at (4,4.5) {$e_3$};
\node at (6,4.5) {$e_4$};
\node at (8,4.5) {$-e_2$};
\node at (10,4.5) {$-e_3$};
\node at (12,4.5) {$-e_4$};

\draw[thick] (0,4) -- (2,0);
\draw[thick] (0,4) -- (4,0);
\draw[thick] (0,4) -- (6,0);

\draw[thick] (2,4) -- (0,0);
\draw[thick] (2,4) -- (4,0);
\draw[thick] (2,4) -- (8,0);

\draw[thick] (4,4) -- (0,0);
\draw[thick] (4,4) -- (6,0);
\draw[thick] (4,4) -- (10,0);

\draw[thick] (6,4) -- (0,0);
\draw[thick] (6,4) -- (2,0);
\draw[thick] (6,4) -- (12,0);

\draw[thick] (8,4) -- (6,0);
\draw[thick] (8,4) -- (8,0);

\draw[thick] (10,4) -- (2,0);
\draw[thick] (10,4) -- (10,0);

\draw[thick] (12,4) -- (4,0);
\draw[thick] (12,4) -- (12,0);

\end{tikzpicture}

%% file: sections/local.tex
\section{Local tropical Poincaré duality spaces}\label{section:local}
In this section, we study \cref{question:local_tpd_when}. In \cref{subsection:face_conditions}, we prove \cref{thm:stars_poincare_duality}. This theorem implies that TPD on faces of a fan, along with vanishing of its tropical BM homology, gives TPD on the whole fan. A version of the proof gives a partial classification of TPD spaces of dimension two. Using \cref{thm:stars_poincare_duality}, we prove \cref{thm:local_tpds_char} in \cref{subsection:local_tpd_theorem}, which states that local TPD spaces are exactly fans whose codimension one faces are local TPD spaces, and all of whose faces have vanishing tropical BM homology. Finally, we use the dimension one classification from \cref{thm:class_dim_1} to give a more geometric characterization of local TPD spaces in \cref{cor:codimension_1_geometric}.

\subsection{TPD from faces}\label{subsection:face_conditions}
We fix a principal ideal domain $R$, and we use the following shortened notation $H_{d,d-p}^{BM}(\Sigma;R) \coloneqq H_d^{BM}(\Sigma, \F_{d-p}^R)$.
We prove \cref{thm:stars_poincare_duality} in two steps: The first step will be to show \cref{prop:diagram_isomorphisms}, which relates the cellular chain complex $C_c^\bullet (\Sigma,\F_R^p)$ to a complex involving the Borel--Moore homology groups $H_d^{BM}(\Star{\gamma}, \F_{d-p}^R)$ for faces $\gamma \in \Sigma$ by using the cap product, which we show is exact. We then prove the theorem by showing that TPD on the faces, along with exactness in the mentioned complex, imply Poincaré duality for the whole fan.

Let $\Sigma$ be a $d$-dimensional rational polyhedral fan. For each maximal face $\alpha \in \Sigma^d$, the constant sheaf $\F_{d-p}^R(\alpha)_{\Cone{\alpha}}$ gives a cochain complex $(C_c^\bullet(\Cone{\alpha},\F_{d-p}^R(\alpha)_{\Cone{\alpha}}), d_\alpha^\bullet)$. Taking the direct sum of these for all $\alpha \in \Sigma^d$, we obtain a complex
\begin{equation}\label{eq:A_cocomplex}
	(A^\bullet,d^\bullet)\coloneqq(\oplus_\alpha C_c^\bullet(\Cone{\alpha},\F_{d-p}^R(\alpha)_{\Cone{\alpha}}), \oplus_\alpha d_\alpha^\bullet).
\end{equation}
The $i$-th term of this complex is given by 
$$A^i = \oplus_\alpha C_c^i (\Cone{\alpha},\F_{d-p}^R(\alpha)_{\Cone{\alpha}})= \oplus_{\alpha \in \Sigma^d} \oplus_{\substack{\gamma \in \Sigma^i \\ \gamma \prec \alpha}} \F_{d-p}^R (\alpha).$$
Rearranging terms, we may use \cref{prop:star_homology_cohomology} to obtain an inclusion
$$\oplus_{\gamma \in \Sigma^i} H_{d,d-p}^{BM}(\Star{\gamma};R) \subseteq A^i.$$
\begin{proposition}\label{prop:lower_row_is_complex}
	There is a cochain complex
	$(\oplus_{\gamma \in \Sigma^\bullet} H_{d,d-p}^{BM}(\Star{\gamma};R), \overline{d^\bullet}),$ which is the restriction of the cochain complex $(A^\bullet, d^\bullet)$ from \cref{eq:A_cocomplex}.
\end{proposition}
\begin{proof}
	It suffices to show that, for each $i\geq 0$, 
	$$d^i\left(\oplus_{\gamma\in \Sigma^i} H_{d,d-p}^{BM}(\Star{\gamma};R)\right) \subseteq \oplus_{\kappa\in \Sigma^{i+1}} H_{d,d-p}^{BM}(\Star{\kappa};R).$$
	This follows from a direct computation.
\end{proof}

\begin{proposition}\label{prop:diagram_isomorphisms}
	For $(\Sigma,w)$ an $R$-balanced fan of dimension $d\geq 2$, there is a  commutative diagram
	\begin{equation}\label{diag:goal}
		\begin{tikzcd}[column sep= small]
			0 & {\F_R^p (v)} & {\bigoplus\limits_{\tau \in \Sigma^1} \F_R^p (\tau)} & {\bigoplus\limits_{\sigma \in \Sigma^2} \F_R^p (\sigma)} & \cdots \\
			0 & { H_{d,d-p}^{BM} (\Sigma;R)} & {\bigoplus\limits_{\tau \in \Sigma^1} H_{d,d-p}^{BM} (\Star{\tau};R)} & {\bigoplus\limits_{\sigma \in \Sigma^2}H_{d,d-p}^{BM} (\Star{\sigma};R)} & \cdots
			\arrow[from=2-1, to=2-2]
			\arrow["{\oplus_\alpha \overline{d_\alpha^0}}", from=2-2, to=2-3]
			\arrow["{\oplus_\alpha \overline{d_\alpha^1}}", from=2-3, to=2-4]
			\arrow["{\frown [\Sigma,w]}", from=1-2, to=2-2]
			\arrow[from=1-1, to=1-2]
			\arrow["{\delta^0}", from=1-2, to=1-3]
			\arrow["{\oplus_\tau \frown [\Star{\tau},w]}", from=1-3, to=2-3]
			\arrow["{\oplus_\alpha \overline{d_\alpha^{2}}}", from=2-4, to=2-5]
			\arrow["{\delta^1}", from=1-3, to=1-4]
			\arrow["{\delta^2}", from=1-4, to=1-5]
			\arrow["{\oplus_\sigma \frown [\Star{\sigma},w]}", from=1-4, to=2-4]
		\end{tikzcd}
	\end{equation}
	with all the vertical maps being injective, where the upper row is given by the complex $(C_c^\bullet(\Sigma,\F_p^R), \delta^\bullet)$, and the lower row is the complex $(\oplus_{\gamma \in \Sigma^\bullet} H_{d,d-p}^{BM}(\Star{\gamma};R), \overline{d^\bullet})$ from \cref{prop:lower_row_is_complex}.
\end{proposition}
\begin{proof}
	First, we wish to show that the following diagram is commutative:
	\[\begin{tikzcd}[column sep = small]
		0 & {\F_R^p (v)} & {\bigoplus\limits_{\tau \in \Sigma^1} \F_R^p (\tau)} & {\bigoplus\limits_{\sigma \in \Sigma^2} \F_R^p (\sigma)} & \cdots \\
		0 & {\bigoplus_{\alpha \in \Sigma^d} \F_{d-p}^R(\alpha)} & {\bigoplus\limits_{\tau \in \Sigma^1}\bigoplus\limits_{\substack{\alpha \in \Sigma^d \\ \alpha \succ \tau}} \F_{d-p}^R(\alpha)} & {\bigoplus\limits_{\sigma \in \Sigma^2}\bigoplus\limits_{\substack{\alpha \in \Sigma^d \\ \alpha \succ \sigma}} \F_{d-p}^R(\alpha)} & \cdots.
		\arrow[from=2-1, to=2-2]
		\arrow["{\oplus_\alpha d_\alpha^0}", from=2-2, to=2-3]
		\arrow["{\oplus_\alpha d_\alpha^1}", from=2-3, to=2-4]
		\arrow["{\oplus_\alpha d_\alpha^{2}}", from=2-4, to=2-5]
		\arrow["{\frown \mathrm{Ch}(\Sigma,w)}", from=1-2, to=2-2]
		\arrow[from=1-1, to=1-2]
		\arrow["{\delta^0}", from=1-2, to=1-3]
		\arrow["{\delta^1}", from=1-3, to=1-4]
		\arrow["{\delta^2}", from=1-4, to=1-5]
		\arrow["{\oplus_\sigma \frown \mathrm{Ch}(\Star{\sigma},w)}", from=1-4, to=2-4]
		\arrow["{\oplus_\tau \frown \mathrm{Ch}(\Star{\tau},w)}", from=1-3, to=2-3]
	\end{tikzcd}\]
	The upper row is the compact support complex $(C_c^\bullet (\Sigma, \F_R^p), \delta^\bullet)$ for the $\F_R^p$ cohomology (see \cref{def:cochain_group}). The lower row is the complex $(A^\bullet, d^\bullet)$ from \eqref{eq:A_cocomplex},
	where the order of indexing is changed for clarity in relation to the cap morphism. 

	The first vertical map in diagram \eqref{diag:goal} is given by the cap product on the chain level of $(\Sigma,w)$, as in \cref{def:cap_product}. For the $r$-th column, the vertical map is given as the direct sum over all $\gamma \in \Sigma^r$ of the maps:
	\begin{align*}
		\frown \mathrm{Ch}(\Star{\gamma},w) \colon \F_R^p (\gamma) &\to \oplus_{\substack{\alpha \in \Sigma^d \\ \alpha \succ \gamma}} \F_{d-p}^R(\alpha) \\
		v &\mapsto  \left(  v_\gamma \, \lrcorner \, w(\alpha)\Lambda_\alpha \right)_{\substack{\alpha \in \Sigma^d \\ \alpha \succ \gamma}}.
	\end{align*}

	To obtain commutativity of the described diagram, we select one square and show commutativity there:
	\begin{equation}\label{diag:commutative_rows}\begin{tikzcd}
		{\bigoplus\limits_{\gamma \in \Sigma^r} \F_R^p (\gamma)} & {\bigoplus\limits_{\kappa \in \Sigma^{r+1}} \F_R^p (\kappa)} \\
		{\bigoplus\limits_{\gamma \in \Sigma^r}\bigoplus\limits_{\substack{\alpha \in \Sigma^d \\ \alpha \succ \gamma}} \F_{d-p}^R(\alpha)} & {\bigoplus\limits_{\kappa \in \Sigma^{r+1}}\bigoplus\limits_{\substack{\alpha \in \Sigma^d \\ \alpha \succ \kappa}} \F_{d-p}^R(\alpha)}
		\arrow["{\oplus_\alpha d_\alpha^r}", from=2-1, to=2-2]
		\arrow["{\oplus_\gamma \frown \mathrm{Ch}(\Star{\gamma},w)}", from=1-1, to=2-1]
		\arrow["{\oplus_\kappa \frown \mathrm{Ch}(\Star{\kappa},w)}", from=1-2, to=2-2]
		\arrow["{\delta^r}", from=1-1, to=1-2]
	\end{tikzcd}\end{equation}
	For $v= (v_\gamma)_{\gamma \in \Sigma^r} \in \bigoplus_{\gamma \in \Sigma^r} \F_R^p (\gamma)$, we can expand the definitions for the right then down composition to get
	\begin{align*}
		((\oplus_\kappa \frown \mathrm{Ch}(\Star{\kappa},w)\circ \delta^r) (v) &= (\oplus_\kappa \frown \mathrm{Ch}(\Star{\kappa},w) ) \left( \left(\textstyle\sum_{\gamma \prec \kappa} \mathcal{O}(\gamma, \kappa) v_\gamma \right)_{\kappa \in \Sigma^{r+1}}  \right) \\
		&= \left(   (\textstyle\sum_{\gamma \prec \kappa} \mathcal{O}(\gamma, \kappa) v_\gamma )\, \lrcorner \, w(\alpha)\Lambda_\alpha  \right)_{\substack{\kappa \in \Sigma^{r+1} \\ \alpha \in \Sigma^{d} \\ \alpha \succ \kappa}}.
	\end{align*}
	For the down then right composition, we get
	\begin{align*}
		((\oplus_{\alpha} d_\alpha^r) \circ ({\oplus_\gamma \frown \mathrm{Ch}(\Star{\gamma}},w))) (v) &= (\oplus_{\alpha} d_\alpha^r) \left( \left( w(\alpha)  v_\gamma \, \lrcorner\, \Lambda_\alpha \right)_{\substack{\gamma \in \Sigma^{r+1}, \; \alpha \in \Sigma^d \\ \alpha \succ \gamma}}  \right) \\
		&= \left(  \textstyle\sum_{\gamma \prec \kappa} \mathcal{O}(\gamma, \kappa)    (v_\gamma \,\lrcorner\, w(\alpha) \Lambda_\alpha)  \right)_{\substack{\kappa \in \Sigma^{r+1} \\ \alpha \in \Sigma^{d} \\ \alpha \succ \kappa}} .
	\end{align*}
	Comparing the two above equations, diagram \eqref{diag:commutative_rows} is commutative since the contraction $\lrcorner\, w(\alpha)\Lambda_\alpha $ is $R$-linear.

	Lastly, we we wish to show injectivity of the vertical maps, when restricting to the Borel--Moore homology groups. By \cref{prop:cap_product_on_stars} and \cref{prop:star_homology_cohomology}, for each $\kappa \in \Sigma$, we have that 
	$$H_d^{BM}(\Star{\kappa}, \F_{d-p}^R) \cong \ker \left(\oplus_{\substack{\alpha \in \Sigma^d \\ \alpha \succ \kappa}} \F_{d-p}^R(\alpha) \to \oplus_{\substack{\beta \in \Sigma^{d-1} \\ \beta \succ \kappa}} \F_{d-p}^R(\beta) \right),$$
	and the given formulas for the maps $\oplus_{\kappa \in \Sigma^{r+1}} \frown \mathrm{Ch}(\Star{\kappa},w)$ correspond exactly to the cap products in homology
	$$\oplus_{\kappa} \frown [\Star{\kappa},w] \colon \bigoplus_{\kappa \in \Sigma^{r+1}} \F_R^p (\kappa) \to \bigoplus\limits_{\kappa \in \Sigma^{r+1}} H_{d}^{BM}(\Star{\kappa}, \F_{d-p}^R).$$
	We have the following diagram when only considering the images
	\begin{equation}
		\begin{tikzcd}[column sep= small]
			0 & {\F_R^p (v)} & {\bigoplus\limits_{\tau \in \Sigma^1} \F_R^p (\tau)} & {\bigoplus\limits_{\sigma \in \Sigma^2} \F_R^p (\sigma)} & \cdots \\
			0 & { H_{d,d-p}^{BM} (\Sigma;R)} & {\bigoplus\limits_{\tau \in \Sigma^1} H_{d,d-p}^{BM} (\Star{\tau};R)} & {\bigoplus\limits_{\sigma \in \Sigma^2}H_{d,d-p}^{BM} (\Star{\sigma};R)} & \cdots
			\arrow[from=2-1, to=2-2]
			\arrow["{\oplus_\alpha \overline{d_\alpha^0}}", from=2-2, to=2-3]
			\arrow["{\oplus_\alpha \overline{d_\alpha^1}}", from=2-3, to=2-4]
			\arrow["{\frown [\Sigma,w]}", from=1-2, to=2-2]
			\arrow[from=1-1, to=1-2]
			\arrow["{\delta^0}", from=1-2, to=1-3]
			\arrow["{\oplus_\tau \frown [\Star{\tau},w]}", from=1-3, to=2-3]
			\arrow["{\oplus_\alpha \overline{d_\alpha^{2}}}", from=2-4, to=2-5]
			\arrow["{\delta^1}", from=1-3, to=1-4]
			\arrow["{\delta^2}", from=1-4, to=1-5]
			\arrow["{\oplus_\sigma \frown [\Star{\sigma},w]}", from=1-4, to=2-4]
		\end{tikzcd}
	\end{equation}
	where the cochain differentials in the lower row have been restricted.

	These vertical maps are the direct sum of cap products, so by \cref{prop:cap_injective}, are injective.
\end{proof}


\begin{proposition}\label{prop:lower_row_exact}
	Let $\Sigma$ be a fan of dimension $d\geq 2$ such that $H_q^{BM}(\Star{\gamma}, \F_p^R)=0$, for $q\not = d$ and all $p$, for each face $\gamma \in \Sigma$. Then the complex $(\oplus_{\gamma \in \Sigma^\bullet} H_{d,d-p}^{BM}(\Star{\gamma};R), \overline{d^\bullet})$ in \cref{prop:lower_row_is_complex} is exact except in the rightmost position.
\end{proposition}
\begin{proof}
	We will now construct a double complex $K_{\bullet,\bullet}$, which corresponds to the Cartan-Eilenberg resolution of diagram \eqref{diag:goal}, using that these are the homology groups of the complexes $C_\bullet^{BM}(\Star{\gamma}, \F_{d-p}^R)$. Then, we will use the two spectral sequences converging to the homology of the total complex $H_\bullet (\mathrm{Tot}(K_{\bullet, \bullet}))$ to deduce that the complex is exact except in the rightmost position.
		
	Let $(K_{\bullet,\bullet},d_\wedge^0, d_>^0)$ be the first-quadrant double complex given by 
	$$K_{r,s} = \bigoplus_{\kappa \in \Sigma^{r}} \bigoplus_{\substack{\gamma \in \Sigma^{d-s} \\ \gamma \succeq \kappa }} \F_{d-p}^R(\gamma),$$
	for $r\geq 0$, $s\geq 0$. 
	Since all the indices used are relating to the dimensions of particular faces of the fan $\Sigma$, this is a first-quadrant double complex.

	The vertical differential $(d_\wedge^0)_{r,s} \colon K_{r,s}\to K_{r,s+1}$ is the direct sum over the differentials $\partial_\bullet^\kappa$ of the chain complex for tropical Borel--Moore homology on the star $\Star{\kappa}$ for each face $\kappa \in \Sigma^r$, i.e. $(d_\wedge^0)_{r,s} = \oplus_{\kappa\in \Sigma^r} \partial_{d-s}^\kappa$ from \cref{prop:star_homology_cohomology}.

	The horizontal differential $(d_>^0)_{r,s} \colon K_{r,s}\to K_{r+1,s}$ is the direct sum over the differentials $d_\gamma^{\bullet}$ in the complex of cochains of compact support for the constant sheaf taking value $\F_{d-p}^R(\gamma)$ on the cone $\Cone{\gamma}$, truncated in degree $1$, for each face $\gamma \in \Sigma^{d-s}$. Explicitly, $(d_>^0)_{r,s} = \oplus_{\gamma \in \Sigma^{d-s}} d_\gamma^{r}$, where $d_\gamma^\bullet$ comes from the complex:
	\begin{equation*}\begin{tikzcd}[column sep = small]
		0 & {\F_{d-p}^R(\gamma)} & {\bigoplus\limits_{\substack{\tau \in \Sigma^1 \\ \gamma \succ \tau}}\F_{d-p}^R(\gamma)} & \dots & {\bigoplus\limits_{\substack{\kappa \in \Sigma^{s-1} \\ \gamma \succ \kappa}}\F_{d-p}^R(\gamma)} & {\F_{d-p}(\gamma)} & 0,
		\arrow[from=1-1, to=1-2]
		\arrow["{d_\gamma^0}", from=1-2, to=1-3]
		\arrow["{d_\gamma^{s-1}}", from=1-5, to=1-6]
		\arrow[from=1-6, to=1-7]
		\arrow["{d_\gamma^{1}}", from=1-3, to=1-4]
		\arrow["{d_\gamma^{s-2}}", from=1-4, to=1-5]
	\end{tikzcd}\end{equation*}
	from \cref{prop:star_homology_cohomology}.

	We have that $d_\wedge^0 \circ d_\wedge^0 = 0$ and $d_>^0 \circ d_>^0 = 0$ since both are direct sums of differentials of complexes.
	Moreover, we have $d_>^0 \circ d_\wedge^0 = d_\wedge^0 \circ d_>^0$ which can be checked directly.

	Now, since $K_{\bullet, \bullet}$ is a double complex, we can look at the two associated spectral sequences converging to the homology of the total complex $(\Tot(K_{\bullet, \bullet}), d_{\text{T}})$ given by $\Tot(K_{\bullet, \bullet})_m = \prod_{r+s=m} K_{r,s}$ with differential $d_{\text{T}} = d_> + d_\wedge$. We refer to \cite[Chapter~5.6]{Weibel} for details. 

	First, we take the spectral sequence $E^r$, with $E^0 = K_{\bullet,\bullet}$ and the first differential $d^0$ being the horizontal differential $d_>^0$ of $K_{\bullet, \bullet}$, which is equivalent to computing the homology row by row.
	Since the rows $K_{\bullet, s}$ are the complexes $\bigoplus_{\gamma \in \Sigma^{d-s}} C_c^\bullet (\Cone{\gamma}, F(\gamma)_{\Cone{\gamma}})$ with $F(\gamma)\coloneqq \F_{d-p}^R(\gamma)$, observing that this is merely the reduced $F(\gamma)$-cohomology of a polytope over which $\gamma$ is a cone, gives
	\begin{equation*}
			H_k (K_{\bullet, s}, d_>^0) \cong 0
	\end{equation*}
	for each $s \neq d$. In degree $d$, we have $H_k (K_{\bullet, d}, d_>^0) \cong \F_{d-p}^R(v)$ for $k=0$ and $H_k (K_{\bullet, d}, d_>^0)=0$ otherwise.
	Thus, the $E^1$ page  becomes
	$$E_{r,s}^1 \cong \begin{cases*} \F_{d-p}^R(v) & $r=0$ and $s=d,$  \\ 0 &otherwise. \end{cases*}$$
	There are now no further non-zero differentials of the spectral sequence, so the $E^1$ page is the $E^\infty$ page. In particular, we conclude that 
	\begin{equation}\label{eq:tot_homology}
		H_q(\Tot(K_{\bullet, \bullet})_\bullet)=
		\begin{cases*}
		\F_{d-p}^R (v) & for $q=d$, \\
		0 &otherwise.
		\end{cases*}
	\end{equation}

	Next, we consider the spectral sequence $\overline{E^r}$, with $\overline{E^0} = K_{\bullet,\bullet}$ and the first differential $d^0$ being the vertical differential $d_\wedge^0$ of $K_{\bullet, \bullet}$. Taking this differential is therefore equivalent to computing the homology column by column. 
	The $r$-th column is the direct sum over each $\gamma \in \Sigma^r$ of the complex
	\[\begin{tikzcd}[column sep = 0.82em]
		0 & {\bigoplus\limits_{\substack{\alpha \in \Sigma^d \\ \alpha \succ \gamma}} \F_{d-p}^R(\alpha)} & {\bigoplus\limits_{\substack{\beta \in \Sigma^{d-1} \\ \beta \succ \gamma}} \F_{d-p}^R(\beta)} & \dots & {\bigoplus\limits_{\substack{\kappa \in \Sigma^{l+1} \\ \kappa \succ \gamma}} \F_{d-p}^R(\kappa)} & {\F_{d-p}^R(\gamma)} & {0,}
		\arrow[from=1-1, to=1-2]
		\arrow["{\partial_n^\gamma}", from=1-2, to=1-3]
		\arrow["{\partial_{n-1}^\gamma}", from=1-3, to=1-4]
		\arrow["{\partial_{l+2}^\gamma}", from=1-4, to=1-5]
		\arrow["{\partial_{l+1}^\gamma}", from=1-5, to=1-6]
		\arrow[from=1-6, to=1-7]
	\end{tikzcd}\]
	each of which has the tropical Borel--Moore homology of the star $\Star{\gamma}$ by \cref{prop:star_homology_cohomology}. Since by assumption, $H_q^{BM}(\Star{\gamma}, \F_p^R)=0$ for $q\not = d$ and all $p$, for each face $\gamma \in \Sigma$, we find:
	\begin{equation*}
		H_k(K_{r,\bullet}, d_\wedge^\bullet) = \begin{cases}
			\oplus_{\kappa\in \Sigma^r} H_d^{BM}(\Star{\kappa}, \F_{d-p}^R) & \text{if $k=0$,}\\
						0 & \text{otherwise.}
					\end{cases}
	\end{equation*}

	Thus the $\overline{E^1}$ page has only the bottom row:
	\[\begin{tikzcd}[column sep = small]
		0 & {H_d^{BM}(\Sigma,\F_{d-p}^R)} & {\bigoplus\limits_{\tau \in \Sigma^1}H_{d}^{BM}(\Star{\tau}, \F_{d-p}^R) } & {\bigoplus\limits_{\sigma \in \Sigma^2}H_{d}^{BM}(\Star{\sigma}, \F_{d-p}^R) } & \dots
		\arrow["{\oplus_\alpha\overline{d_\alpha^0}}", from=1-2, to=1-3]
		\arrow["{\oplus_\alpha\overline{d_\alpha^1}}", from=1-3, to=1-4]
		\arrow["{\oplus_\alpha\overline{d_\alpha^2}}", from=1-4, to=1-5]
		\arrow[from=1-1, to=1-2]
	\end{tikzcd}\]
	The $\overline{E^2}$ page is then merely the homology of this complex, and since it is concentrated in one row, this must be the $\overline{E^\infty}$ page. In particular, by \eqref{eq:tot_homology}, the complex only has homology in the rightmost position.
\end{proof}
\begin{theorem}\label{thm:stars_poincare_duality}
	Let $R$ be a principal ideal domain, and $(\Sigma,w)$ be an $R$-balanced fan of dimension $d\geq 2$, with $H_{q}^{BM}(\Sigma, \F_{p}^R) = 0$ for $q\not = d$, for all $p$.
	If $(\Star{\gamma},w)$ satisfies TPD over $R$, for each  $\gamma\in \Sigma$ with $\Star{\gamma} \neq \Sigma$, then $(\Sigma,w)$ satisfies TPD over $R$.
\end{theorem}
\begin{proof}
	By assumption $H_{q}^{BM}(\Sigma, \F_{d-p}^R) = 0$ for $q\neq d$, and $H^q (\Sigma,\F_R^{p})=0$ for $q\neq 0$ by \cref{remark:trivial_hom_cohom}, for all $p$. Thus the cap product map $$\frown[\Sigma,w] \colon H^q (\Sigma,\F_R^{p})\to H_{d-q}^{BM}(\Sigma, \F_{d-p}^R)$$
	is an isomorphism for all $q=1,\dots,d$, for all $p$, and it remains to show that $$\frown[\Sigma,w] \colon H^0 (\Sigma,\F_R^{p})\to H_{d}^{BM}(\Sigma, \F_{d-p}^R)$$
	is an isomorphism for all $p$. 

	Since $H_{q}^{BM}(\Sigma, \F_{d-p}^R) = 0$ for $q\neq d$ and $(\Star{\gamma},w)$ satisfies TPD over $R$ for each $\gamma\in \Sigma$ with $\Star{\gamma} \neq \Sigma$ gives that $H_{d-q}^{BM}(\Star{\gamma}, \F_{d-p}^R) = 0$ for all $\gamma \in \Sigma$. Thus, by \cref{prop:lower_row_exact}, the lower row in diagram \eqref{diag:goal} is exact in all degrees except $d$. 
	
	Moreover, the upper row of diagram \eqref{diag:goal} is the complex $C_c^\bullet (\Sigma, \F_R^p)$, which can be seen to be the dual complex to $C_\bullet^{BM}(\Sigma,\F_p^R)$ by the definitions (\cref{def:cochain_group,def:chain_group,def:Fp_sheaves}). 
	
	The complex $C_\bullet^{BM}(\Sigma,\F_p^R)$ consists only of free $R$-modules, since $\F_p^R(\gamma)$ is a sublattice of $N\otimes_\Z R$ for all $\gamma \in \Sigma$, and the ring $R$ is a principal ideal domain, hence we may apply the Universal Coefficient Theorem for cohomology \cite[Theorem 3.6.5]{Weibel}. Thus, for each $q$, we have:
	\[\begin{tikzcd}[column sep = small]
		0 & {\Ext_R (H_{q-1}^{BM} (\Sigma,\F_p^R),R)} & {H_c^q (\Sigma,\F_p^R)} & {\Hom_R (H_{q-1}^{BM} (\Sigma,\F_p^R),R)} & 0.
		\arrow[from=1-1, to=1-2]
		\arrow[from=1-2, to=1-3]
		\arrow[from=1-3, to=1-4]
		\arrow[from=1-4, to=1-5]
	\end{tikzcd}\]
	Since we assumed $H_{d-q}^{BM}(\Sigma, \F_{d-p}^R) = 0$ for $q\not = 0$, for all $p$, one has $H_c^q (\Sigma,\F_p^R)=0$ for $q\not = d$, for all $p$. Hence the upper row of diagram \eqref{diag:goal} is exact except in the last position.

	The cokernel complex to the chain complex map in \eqref{diag:goal} gives following short exact sequence of chain complexes:
	$$0 \to (C_c^\bullet(\Sigma,\F_p^R), \delta^\bullet) \to \left(\oplus_{\gamma \in \Sigma^\bullet} H_d^{BM} (\Star{\gamma}, \F_{d-p}^R),\; \oplus_\alpha \overline{d_\alpha^\bullet} \right) \to \coker (\frown) \to 0.$$
	This gives a long exact sequence in homology, and since the two first chain complexes are exact in all but the last position, so is the cokernel chain complex. Thus we have the following exact sequence:
	\[\begin{tikzcd}[column sep=small]
		0 & {\coker(\frown [\Sigma,w])} & {\coker(\oplus_\tau \frown [\Star{\tau},w])} & {\coker(\oplus_\sigma \frown [\Star{\sigma},w])} & \cdots
		\arrow[from=1-2, to=1-3]
		\arrow[from=1-1, to=1-2]
		\arrow[from=1-3, to=1-4]
		\arrow[from=1-4, to=1-5]
	\end{tikzcd}\]
	Since each of the stars $(\Star{\gamma},w)$ satisfies TPD over $R$, we have $\coker(\oplus_\gamma \frown [\Star{\gamma},w])=0$ and so exactness gives ${\coker(\frown [\Sigma,w])}=0$. Thus $\frown [\Sigma,w]$ is both injective by \cref{prop:cap_injective} and surjective, hence is an isomorphism.
\end{proof}
\begin{remark}
	In the proof of \cref{thm:stars_poincare_duality}, the condition that $R$ is a PID is only used to show that $H_{q}^{BM}(\Sigma,\F_p^R)=0$ for all $q\neq d$ implies that $H_c^q(\Sigma,\F_R^p)=0$ for all $q\neq d$. One can let $R$ be an arbitrary commutative ring if we instead assume this latter condition, giving another variant of the theorem.
\end{remark}
It is not sufficient that all the star fans $\Star{\gamma}$ of faces $\gamma \in \Sigma$ with $\gamma\neq v$ satisfy TPD. The assumption $H_{d-q}^{BM}(\Sigma, \F_{d-p}^R) = 0$ for $q\not = 0$, for all $p$, is necessary and not implied by TPD of the faces. This is shown by the following example, which is also studied in \cite{mastersthesis} and in \cite[Section 11.1]{AminiPiquerezFans}.
\begin{example}\label{ex:tpd_faces_not_tpd}
	\begin{figure}
		\includegraphics[width=\textwidth]{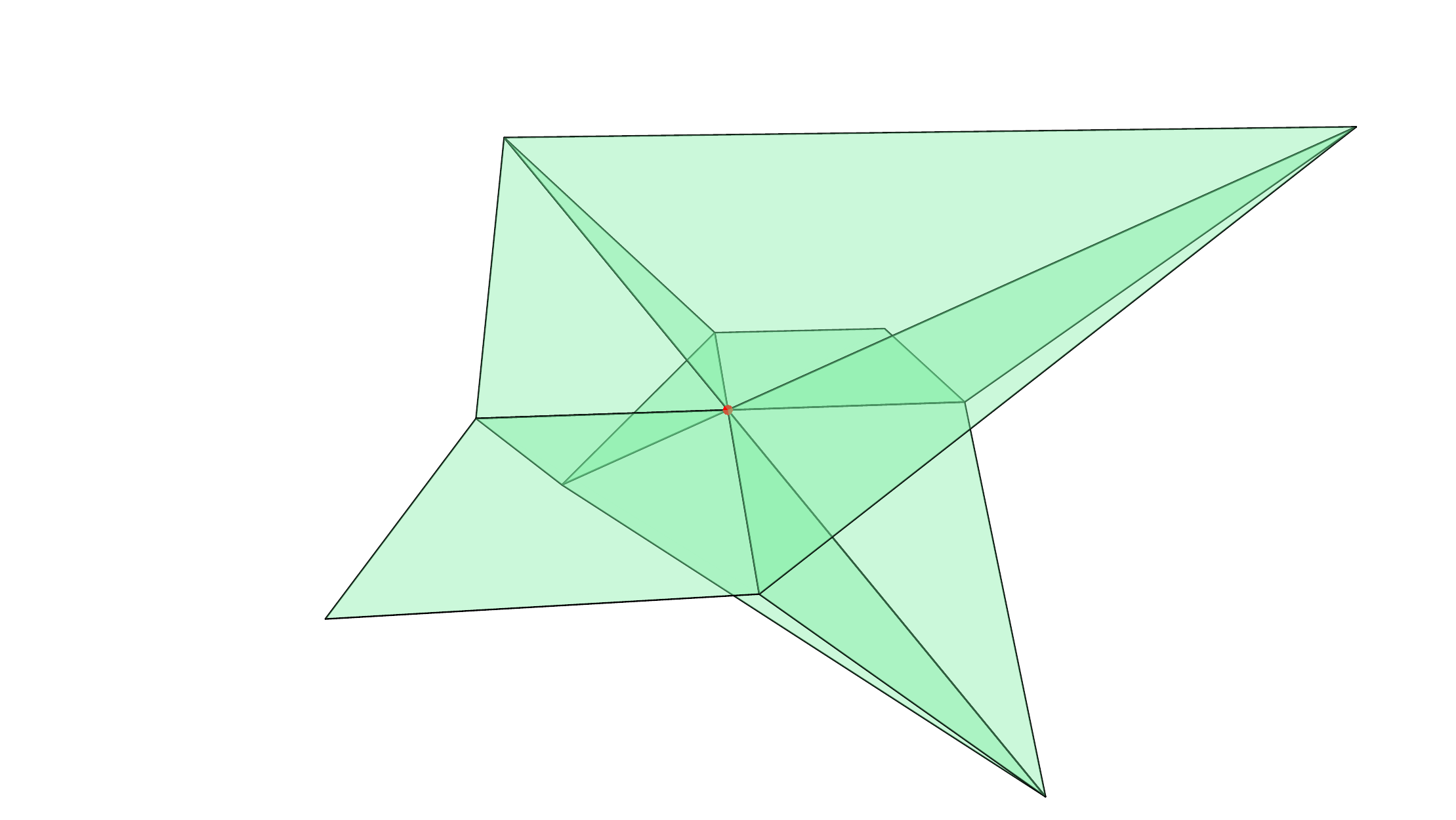}
		\caption{A fan whose faces excluding the vertex satisfy TPD, but does not itself satisfy it. Figure generated using \cite{polymake}}
		\label{fig:balanced_faces_not_tpd}
	\end{figure}
	Let $\Sigma$ be the fan shown in \cref{fig:balanced_faces_not_tpd}, where the rays are 
	$e_1, e_2, -e_1, -e_2, -e_1+e_2 +e_3, -e_1+e_2-e_3, e_1-e_2 +e_3, e_1-e_2 -e_3$. Each of its two-dimensional faces is maximal, so that the star at these faces is just a two-dimensional linear space, which satisfies TPD. Moreover, each ray has exactly three faces meeting in it, so that the stars are uniquely balanced, and satisfy TPD. We could therefore expect \cref{thm:stars_poincare_duality} to give us that the whole fan $\Sigma$ has TPD. 
	
	However, observe that $\dim_\Q \F_2^\Q (\sigma)=1$ for each $\sigma\in\Sigma^2$, while $\dim_\Q \F_2^\Q (\tau)=2$ for each $\tau\in \Sigma^1$ and $\dim_\Q \F_2^\Q (v)=3$. There are $12$ two-dimensional faces and $8$ one-dimensional faces, so that 
	$$\chi(C_\bullet^{BM}(\Sigma,\F_2^\Q))=12-8\cdot 2 +3 =-1.$$
	Since $\chi(C_\bullet^{BM}(\Sigma,\F_2^\Q))=\chi(H_\bullet^{BM}(\Sigma,\F_2^\Q))$, we must have that $H_{1}^{BM}(\Sigma, \F_2^{\Q}) \not \cong 0$. Finally, since $H^{1}(\Sigma, \F_\Q^0)=0$, the cap product $\frown [\Sigma,w] \colon H^{1}(\Sigma, \F_\Q^0) \to H_{1}^{BM}(\Sigma, \F_{2}^\Q)$ cannot be an isomorphism.
\end{example}
\begin{proposition}\label{prop:class_dim_two}
	Let $\mathbbm{k}$ be a field, and $(\Sigma,w)$ a $\mathbbm{k}$-balanced fan of dimension $2$. Suppose $H_{q}^{BM}(\Sigma,\F_{p}^\mathbbm{k}) =0$ for $q\neq 2$, for all $p$.
	Then $\Sigma$ satisfies TPD over $\mathbbm{k}$ if and only if each of the stars $(\Star{\tau},w)$, for $\tau \in \Sigma^1$ satisfies TPD over $\mathbbm{k}$.
\end{proposition}
\begin{proof}
	First, we show that for each $\sigma\in \Sigma^2$, the star $\Star{\sigma}$ satisfies TPD over $\mathbbm{k}$. For each $\sigma \in \Sigma^2$, we have from \cref{prop:star_homology_cohomology} that
	\begin{align*}
		H^0 (\Star{\sigma}, \F_\mathbbm{k}^p) &= \F_\mathbbm{k}^p (\sigma) = \left(\textstyle{\bigwedge^p} L_\Z (\sigma)\otimes_\Z \mathbbm{k}  \right)^*, \\
		H_d^{BM} (\Star{\sigma}, \F_{d-p}^\mathbbm{k}) &= \F_{d-p}^\mathbbm{k} (\sigma)= \textstyle{\bigwedge^{d-p}} L_\Z (\sigma)\otimes_\Z \mathbbm{k}.
	\end{align*}
	Moreover, the cap product is an injective map by \cref{prop:cap_injective}. These two vector spaces have the same dimension, hence the cap product is an isomorphism.
	
	Next, we consider again the sequence 
	\[\begin{tikzcd}[column sep=small]
		0 & {\coker(\frown [\Sigma,w])} & {\coker(\oplus_\tau \frown [\Star{\tau},w])} & {\coker(\oplus_\sigma \frown [\Star{\sigma},w])} & \cdots
		\arrow[from=1-2, to=1-3]
		\arrow[from=1-1, to=1-2]
		\arrow[from=1-3, to=1-4]
		\arrow[from=1-4, to=1-5]
	\end{tikzcd}\]
	from the proof of \cref{thm:stars_poincare_duality}. Since $\Star{\sigma}$ satisfies TPD over $k$, $\coker(\oplus_\sigma \frown [\Star{\sigma},w])=0$, and so $\coker(\frown [\Sigma,w]) \cong \coker(\oplus_\tau \frown [\Star{\tau},w])$. Since both these maps are injective by \cref{prop:cap_injective}, the result follows.
\end{proof}
\begin{remark}	
	\cref{thm:stars_poincare_duality} shows that under the assumption of the vanishing of Borel--Moore homology, TPD on a fan $\Sigma$ can be deduced from TPD on its faces. In fact, it is not necessary to assume that all the faces satisfy TPD:
	In general, the ``vertical first'' spectral sequence in \cref{prop:lower_row_exact} degenerates on page 2 when $H_q^{BM}(\Star{\gamma}, \F_p^R)=0$, for $q\not = d$ and all $p$, for each face $\gamma \in \Sigma$. However, the exactness of the lower row in diagram \eqref{diag:goal} in positions 0 and 1 follow from the weaker assumption that $H_q^{BM}(\Star{\gamma}, \F_p^R)=0$, for $q\not = d$ and all $p$, for each face $\gamma \in \Sigma^i$, with $i=0,1$. One can then show, in a restricted version of the proof of \cref{thm:stars_poincare_duality}, that TPD for all $\tau \in \Sigma^1$ implies that $\Sigma$ satisfies TPD.
\end{remark}
	
\subsection{A characterization of local TPD spaces}\label{subsection:local_tpd_theorem}
We now turn to studying fans which satisfy TPD at every face. Using \cref{thm:stars_poincare_duality}, we characterize such fans as the ones for which the condition holds in codimension 1 along with a vanishing condition on Borel--Moore homology, which was suggested to us by Amini and Piquerez \cite{AminiPiquerezFans}.

\begin{definition}\label{def:local_TPD_space}
	Let $R$ be a ring, and $(\Sigma,w)$ an $R$-balanced fan. If, for each face $\gamma \in \Sigma$, the star fan $\Star{\gamma}$ satisfies tropical Poincaré duality over $R$, we say that $\Sigma$ is a \emph{local tropical Poincaré duality space over $R$}.
\end{definition}

In particular, this implies that $\Sigma$ satisfies TPD over $R$. In the case where $R=\Z$ or $\Q$, being a local TPD space is equivalent to the \emph{tropical smoothness} introduced by Amini and Piquerez in \cite{AminiPiquerezFans}. We use a different notion of the star of a face, but the equivalence of the definitions can be seen from \cite[Proposition 3.17]{AminiPiquerezFans}, which in turn follows from the tropical Künneth theorem \cite[Theorem B]{GrossShokriehSheaf}.

\begin{theorem}\label{thm:local_tpds_char}
	Let $R$ be a principal ideal domain, and $(\Sigma,w)$ a $d$-dimensional $R$-balanced fan. Then $\Sigma$ is a local TPD space over $R$ if and only if $H_q^{BM}(\Star{\gamma}, \F_{p}^R)=0$ for all $\gamma\in \Sigma$ and $q\neq d$, and for all faces $\beta$ of codimension 1, the star fans $\Star{\beta}$ are TPD spaces over $R$.
\end{theorem}
\begin{proof}
	If $\Sigma$ is a local TPD space over $R$, each of the star fans $\Star{\gamma}$ for $\gamma \in \Sigma$, in particular the codimension 1 faces are TPD spaces over $R$. Moreover, this implies that the
	Borel–Moore homology groups $H_q^{BM}(\Star{\gamma},\F_p^R)$ vanish for $q\neq d$ and all $\gamma \in \Sigma$.
	
	Next, assume that the star fans $\Star{\beta}$ for $\beta\in \Sigma^{d-1}$ are local TPD spaces over $R$ and $H_q^{BM}(\Star{\gamma}, \F_{p}^R)=0$ for all $\gamma\in \Sigma$. First, we apply \cref{thm:stars_poincare_duality} to all faces of codimension two $\mu$ in $\Sigma$. For a given $\mu\in \Sigma^{d-2}$, we have $H_q^{BM}(\Star{\mu}, \F_{p}^R)=0$ for all $p$ by assumption. Moreover, each face $\widetilde{\beta}\in \Star{\mu}$ is a subdivision of a face $\beta\in \Sigma$ with $\beta\prec \mu$. By assumption all these faces of $\Sigma$ are TPD spaces, and therefore the subdivided faces $\widetilde{\beta}$ of $\Star{\mu}$ are as well. Hence we may apply \cref{thm:stars_poincare_duality}, and conclude that $\Star{\mu}$ is a TPD space. Thus, all of codimension 2 of $\Sigma$ are TPD spaces. Proceeding inductively, we can apply \cref{thm:stars_poincare_duality} to all the stars $\Star{\gamma}$ of faces $\gamma \in \Sigma$. Thus $\Sigma$ is a local TPD space.
\end{proof}
\begin{corollary}\label{cor:codimension_1_geometric}
	Let $(\Sigma, w)$ a $d$-dimensional $\Z$-balanced fan. Then $\Sigma$ is a local TPD space over $\Z$ if and only if $H_q^{BM}(\Star{\gamma},\F_p^\Z) = 0$ for $q\neq d$ and all $\gamma \in \Sigma$, all the weights are $\pm 1$, and for all faces $\beta$ of codimension 1, the star fans $\Star{\beta}$ are uniquely $\Z$-balanced in the sense of \cref{def:balanced}.
\end{corollary}
\begin{proof}
	For each face $\beta$ of codimension 1 of $\Sigma$, observe that each face of dimension $d$ of the star fan $\Star{\beta}$ is a linear space, hence is a TPD space over $\Z$. Furthermore, each star fan $\Star{\beta}$ has a $(d-1)$-dimensional lineality space, and we may write $\Star{\beta} = \Sigma_\beta \times \R^{d-1}$ for some ``reduced star'' $\Sigma_\beta$ of dimension 1. Since $\R^{d-1}$ satisfies TPD over $\Z$, by \cite[Proposition 3.18]{AminiPiquerezFans}, the star fan $\Star{\beta}$ is a TPD space over $\Z$ if and only if $\Sigma_\beta$ is a a TPD space over $\Z$. By \cref{thm:class_dim_1}, this is the case if and only if $\Sigma_\beta$ is uniquely $\Z$-balanced with $\pm 1$-weights. Therefore $\Star{\beta}$ is a local TPD space over $\Z$ if and only if it is uniquely $\Z$-balanced with $\pm 1$-weights. 
	Finally, the equivalence follows from comparing with \cref{thm:local_tpds_char}.
\end{proof}
\begin{remark}
	Passing from \cref{thm:local_tpds_char} to \cref{cor:codimension_1_geometric} is mostly dependent on the Künneth formula for the $\F_p^\Z$ cosheaves from \cite{GrossShokriehSheaf}. A generalization of the Künneth formula to $\F_p^R$ for another ring $R$ would also lead to a generalization of \cref{cor:codimension_1_geometric}.
\end{remark}
\cref{thm:local_tpds_char} illustrates that it would be desirable to obtain a geometric understanding of the vanishing condition for the tropical Borel–Moore homology.
\begin{question}[Geometry of BM homology vanishing?]\label{question:BM_vanishing}
	Let $(\Sigma,w)$ be an $R$-balanced $d$-dimensional fan. Can the fans with $H_q^{BM}(\Star{\gamma}, \F_{p}^R)=0$ for each face $\gamma\in \Sigma$, $q\neq d$ and all $p$ be geometrically characterized?
\end{question}
We note that it is not clear whether TPD of the whole fan implies local TPD. We have not been able to construct a fan satisfying TPD such that the star of one of its faces does not.
\begin{question}[Global versus Local TPD]\label{question:global_vs_local_fan}
	Let $(\Sigma,w)$ be an $R$-balanced fan which satisfies TPD over $R$. Does $\Star{\gamma}$ also satisfy TPD over $R$ for each $\gamma \in \Sigma$?
\end{question}
Even assuming that $H_q^{BM}(\Star{\gamma}, \F_{p}^R)=0$ for $q\neq d$, along with TPD on the whole fan $(\Sigma,w)$, the proof of \cref{thm:stars_poincare_duality} does not directly imply that $\Sigma$ is a local TPD space generally.
In algebraic topology, the question of going from Poincaré duality globally on a CW complex to Poincaré duality locally has been investigated using techniques from surgery and homotopy theory (see \cite{RanickiPD} for an overview).

%% file: sections/polyhedral_spaces.tex
\section{Tropical Poincaré duality for polyhedral spaces}\label{section:poly_complexes}
In this section, we use the results of \cref{section:local} to prove that abstract tropical cycles which have charts to local TPD spaces satisfy tropical Poincaré duality.
In \cite[Theorem 5.3]{Lefschetz11}, the Mayer--Vietoris argument that shows that tropical manifolds satisfy tropical Poincaré duality over $\Z$ is predicated on the existence of charts to fans of matroids, which are local TPD spaces over $\Z$.
This suggests that the local TPD spaces characterized in \cref{thm:local_tpds_char} are useful as building blocks for general spaces satisfying TPD. We show this in \cref{thm:local_tpd_polyspaces_have_tpd}.
We refer to \cite{Lefschetz11} for the definitions of rational polyhedral spaces, rational polyhedral structures, as well as the tropical cohomology and Borel--Moore homology theories available on such spaces. Here we generalize these to take coefficients in an arbitrary commutative ring $R$, as in \cref{def:tropical_homology_cohomology}.
Moreover, one can generalize \cite[Definitions 4.7-4.8]{Lefschetz11} of the weight function to an arbitrary commutative ring, as in \cref{def:weights_R}, which gives rise to a fundamental chain $\mathrm{Ch}(X,w)\in C_d^{BM}(X,\F_d^R)$, for $d=\dim X$.
\begin{definition}\label{def:abstract_tropical_cycle}
	A rational polyhedral space $X$ of pure dimension $d$ with a rational polyhedral structure $\Ccal$ and a weight function $w$ is \emph{balanced} if the fundamental chain $\mathrm{Ch}(X,w)\in C_d^{BM}(X,\F_d^R)$ is closed, inducing a \emph{fundamental class} $[X,w]\in H_d^{BM}(X,\F_d^R)$ in tropical Borel--Moore homology. We call the triple $(X,\Ccal,w)$ an \emph{abstract tropical $R$-cycle}.
\end{definition}
Abstract tropical $R$-cycles are the candidate spaces for \emph{satisfying tropical Poincaré duality over $R$}, slightly generalizing \cite[Definition 4.11]{Lefschetz11}.
\begin{definition}
	Let $X$ be an abstract tropical $R$-cycle of dimension $d$. The fundamental class $[X,w]$ induces a \emph{cap product}
	$$\frown [X,w] \colon H^{q}(X,\F_R^p) \to H_{d-p}^{BM}(X,\F_{d-p}^R)$$
	between tropical cohomology and tropical Borel--Moore homology. If these are isomorphisms for all $p$ and $q$, we say that $X$ is a \emph{tropical Poincaré duality space over $R$}.
\end{definition}
\begin{definition}\label{def:local_tpd_polyspace}
	Let $(X, \Ccal,w)$ be an abstract tropical $R$-cycle over a commutative ring $R$. We say that $(X, \Ccal,w)$ is a \emph{local tropical Poincaré duality space}  if for each $\sigma \in \Ccal$, the rational polyhedral complexes $\set{\phi_\sigma(\tau)\mid \tau \in \Star{\sigma}}$ are local TPD spaces over $R$.
\end{definition}
\begin{example}
	Tropical manifolds, which have charts to Bergman fans of matroids, are examples of local TPD spaces over $\Z$ and $\R$, see \cite{Lefschetz11,JellShawSmacka}.
\end{example}
\begin{theorem}\label{thm:local_tpd_polyspaces_have_tpd}
	Let $X$ be a local tropical Poincaré duality space over a principal ideal domain $R$. Then $X$ satisfies  tropical Poincaré duality over $R$.
\end{theorem}
\begin{proof}
	The two steps of the proof given in \cite[Proof of Theorem 5.3]{Lefschetz11} carry through. Since the open stars of faces satisfy TPD over $R$, the first step is identical, noting that the same arguments carry through working in the category of $R$-modules. The induction argument given in the second step also carries through, as the same sequence of complexes can be constructed in the category of $R$-modules.
\end{proof}
\begin{remark}
	Note that \cref{def:local_tpd_polyspace} in the case where $R=\Z$ is equivalent to the definition of \emph{smooth tropical variety} given in \cite[Definition 3.22]{AminiPiquerezFans}, such that the case $R=\Z$ of \cref{thm:local_tpd_polyspaces_have_tpd} is equivalent to \cite[Theorem 3.23]{AminiPiquerezFans}.
\end{remark}

\cref{thm:local_tpd_polyspaces_have_tpd} justifies the naming in \cref{def:local_tpd_polyspace}, generalizing the relationship between local TPD spaces and TPD spaces as defined in \cref{def:local_TPD_space} and \cref{def:tpd}. Moreover, \cref{question:global_vs_local_fan} about the relationship between local TPD and TPD are also applicable in this more general setting.
\begin{question}[Global versus Local TPD for abstract tropical cycles]\label{question:global_vs_local_polyspace}
	Let $(X,\Ccal,w)$ be an abstract tropical $R$-cycle satisfying TPD over $R$. Does $\Star{\gamma}$ also satisfy TPD over $R$ for each $\gamma \in \Ccal$?
\end{question}